\documentclass[12pt]{amsart}

\usepackage{amsmath,amssymb,amsthm,enumerate,xcolor}

\usepackage[margin=1in]{geometry}

\usepackage[utf8]{inputenc}
\usepackage[T1]{fontenc}
\usepackage{lmodern}
\usepackage{hyperref}
\usepackage[capitalise]{cleveref}
\usepackage{enumitem}

\newtheorem{theorem}{Theorem}
\newtheorem{lemma}[theorem]{Lemma}
\newtheorem{corollary}[theorem]{Corollary}
\newtheorem{proposition}[theorem]{Proposition}
\theoremstyle{remark} \newtheorem*{remark}{Remark}
\theoremstyle{remark} 
\theoremstyle{definition} 

\numberwithin{equation}{section}
\numberwithin{theorem}{section}

\newcommand{\Q}{\mathbb{Q}}
\newcommand{\Res}{\mathop{\mathrm{Res}}}
\newcommand{\N}{\mathrm{N}}
\newcommand{\kp}{\mathfrak{p}}
\renewcommand{\epsilon}{\varepsilon}

\renewcommand{\leq}{\leqslant}
\renewcommand{\geq}{\geqslant}

\title{Effective Brauer--Siegel theorems for  Artin $L$-functions}
\author{Peter Jaehyun Cho}
\address{Department of Mathematical Sciences, Ulsan National Institute of Science and\newline
	\indent Technology, UNIST-gil 50, Ulsan 44919, Korea}
\email{petercho@unist.ac.kr}

\author{Robert J. Lemke Oliver}
\address{Department of Mathematics, University of Wisconsin-Madison, Madison, WI 53706, USA\ \indent Department of Mathematics, Tufts University, Medford, MA 02155, USA}
\email{lemkeoliver@wisc.edu}

\author{Asif Zaman}
\address{Department of Mathematics, University of Toronto, Toronto, ON M5S 2E4, CANADA}
\email{asif.zaman@utoronto.ca}

\begin{document}

	\maketitle

\begin{abstract}
	Given a number field $K \neq \Q$, in a now classic work, Stark pinpointed the possible source of a so-called Landau--Siegel zero of the Dedekind zeta function $\zeta_K(s)$ and used this to give effective upper and lower bounds on the residue of $\zeta_K(s)$ at $s=1$.
	We extend Stark's work to give effective upper and lower bounds for the leading term of the Laurent expansion of general Artin $L$-functions at $s=1$ that are, up to the value of implied constants, as strong as could reasonably be expected given current progress toward the generalized Riemann hypothesis.  Our bounds are completely unconditional, and rely on no unproven hypotheses about Artin $L$-functions.
\end{abstract}

%
%
%
%	SECTION 1: INTRODUCTION
%
%
%

\section{Introduction} \label{sec:intro}
For any number field $K \neq \Q$ with $D_K = |\mathrm{Disc}(K)|$ and any $\epsilon > 0$, it is classical knowledge that the residue of its Dedekind zeta function $\zeta_K(s)$ at $s=1$ satisfies 
\[
D_K^{-\epsilon} \ll_{[K:\Q],\epsilon} \Res_{s=1} \zeta_K(s) \ll_{[K:\Q]} (\log D_K)^{[K:\Q]-1}. 
\]
The upper bound is due to Landau \cite{Landau} and the implied constant is effectively computable. The lower bound is famously known as the Brauer--Siegel theorem \cite{Brauer1947,Siegel1935}, but unfortunately the proof produces an ineffective implied constant for arbitrary fields $K$ and $\epsilon > 0$. This defect was a serious problem in many applications until, in a breakthrough 1974 paper, Stark \cite{Stark} made three fundamental contributions. 

First, for any number field $K \neq \Q$, Stark showed that the Dedekind zeta function $\zeta_K(s)$ has at most one zero  in the region
\[
\Re(s) > 1 - \frac{1}{4 \log D_K}, \quad |\Im(s)| < \frac{1}{4 \log D_K},
\]
and, if this zero $\beta_K$ exists, then $\beta_K$ is real and simple.  We refer to $\beta_K$, if it exists, as the \textit{exceptional zero} (or the \textit{Landau--Siegel zero}) of $K$. Conjecturally, $\beta_K$ does not exist. Second, building on   Heilbronn \cite{Heilbronn}, Stark showed that if additionally $K/k$ is a normal extension with Galois group $G = \mathrm{Gal}(K/k)$, then there exists a unique irreducible character $\psi_{K/k} \in \mathrm{Irr}(G)$ such that $\psi_{K/k}^2 = \mathbf{1}_G$, the trivial character of $G$, and $\beta_K$ is a real simple zero of the (1-dimensional) Artin $L$-function $L(s,\psi_{K/k})$. We refer to $\psi_{K/k}$, if it exists,  as the \textit{exceptional character} of   $K/k$. Third, using effective cases of the Brauer--Siegel theorem, Stark provided an effective lower bound on $1-\beta_K$ and hence on the residue of $\zeta_K(s)$ at $s=1$, proving that  
\[
\Res_{s=1} \zeta_K(s) \gg_{[K:\Q]} D_K^{-1/[K:\Q]}
\]
for any number field $K \neq \Q$ with an effective implied constant.  
	
In this paper, we extend Stark's work to all Artin $L$-functions for any Galois extension $K/k$ of number fields with Galois group $G$. All implied constants will be effectively computable.  Indeed, for any character $\chi$ of $G$, we will provide effective upper and lower bounds for the leading term in the Laurent expansion of Artin $L$-functions $L(s,\chi)$ at $s=1$. More precisely, it is known by classical work of Artin and Hecke that the function $(s-1)^{\langle \chi, \mathbf{1}_G \rangle} L(s,\chi)$ is holomorphic and non-zero at $s=1$, where $\langle \, \cdot \, , \, \cdot \, \rangle$ is the inner product on $G$. Our goal is to estimate the non-zero complex number  
\begin{equation} \label{eqn:kappa-def}
	\kappa(\chi) := \lim_{s \to 1} (s-1)^{\langle \chi, \mathbf{1}_G \rangle} L(s,\chi)		
\end{equation}
in terms of standard invariants, such as the degree $\chi(1)$, the analytic conductor $q(\chi)$  of the Artin $L$-function $L(s,\chi)$, the absolute value of the discriminant $D_K$,  and 
	\begin{equation}
		\label{eqn:mu-def}
		\mu(\chi)
			:= \min\{ \Re(\chi(g)) : g \in G\}. 
	\end{equation}
Note $-\chi(1) \leq \mu(\chi) \leq \chi(1)$ always and, if $\langle \chi, \mathbf{1}_G \rangle = 0$, then    $\mu(\chi) < 0$ and $\kappa(\chi) = L(1,\chi)$. Some quantities will  be stated in terms of the induction\footnote{More precisely, $\widetilde{\chi}$ is the induction to $\Q$ of the pullback of $\chi$ via the quotient map $\mathrm{Gal}(\widetilde{K}/k) \to \mathrm{Gal}(K/k)$.} of $\chi$ to the Galois closure $\widetilde{K}$ of $K$ over $\mathbb{Q}$, denoted by the character $\widetilde{\chi}$ of $\widetilde{G} := \mathrm{Gal}(\widetilde{K}/\Q)$.  

 Our first main result estimates $\kappa(\chi)$ in terms of $D_K$ and the exceptional character $\psi_{K/k}$. 

\begin{theorem} \label{thm:disc}
	Let $K/k$ be a Galois extension of number fields with Galois group $G$, and let $\psi_{K/k}$ be the exceptional character, if it exists. For any character $\chi$ of $G$, we have that
	\[
	|\kappa(\chi)| \ll_{[k:\Q],|G|,\chi(1)} (\log D_K)^{\widetilde{\chi}(1) - \langle \chi, \mathbf{1}_G \rangle}
	\]
	and
	\[
	|\kappa(\chi)| \gg_{[k:\Q],|G|,\chi(1)}  \Big( \frac{\log D_K}{D_K^{1/[K:\Q]} } \Big)^{\nu(\chi)}(\log D_K)^{\mu(\widetilde{\chi}) - \langle \chi, \mathbf{1}_G \rangle},
	\]
	where $\nu(\chi)=\langle \chi, \psi_{K/k} \rangle$ if $\psi_{K/k}$ exists and $\nu(\chi) = 0$ otherwise. 
\end{theorem}

 Our second main result estimates $\kappa(\chi)$ in terms of the conductor $q(\chi)$, which can be advantageous when $q(\chi)$ is a small power of $D_K$. This situation arises naturally in applications since the conductor discriminant formula implies 
 \[
 D_K = \prod_{\psi} q(\psi)^{\psi(1)} \quad \text{ and } \quad q(\chi)  = \prod_{\psi} q(\psi)^{\langle\chi, \psi\rangle},
 \]
 where $\psi$ runs over $\mathrm{Irr}(G)$, the set of irreducible characters of $G$. 
 However, the estimation of $\kappa(\chi)$ in terms of $q(\chi)$ turns out to be more subtle, because we must account for the subset  $\Psi_{K/k}(G)$ of \textit{potentially exceptional  characters} associated to $K/k$, all of which are trivial or quadratic (and hence satisfy Artin's holomorphy conjecture). This subset is defined as 
\begin{equation} \label{eqn:exceptional-quadratics} 	
\Psi_{K/k}(G) := \Big\{ \psi \in \mathrm{Irr}(G) : \psi^2 = \mathbf{1}_G \text{ and } L(s,\psi) \text{ has a real zero } \beta_{\psi} > 1 - \frac{1}{4 \log q(\psi)} \Big\}. 
\end{equation}

\begin{theorem} \label{thm:cond}
Let $K/k$ be a Galois extension of number fields with Galois group $G$. For any character $\chi$ of $G$, we have that
\[
 |\kappa(\chi)| \ll_{[k:\Q],|G|,\chi(1)} (\log eD_k)^{\widetilde{\chi}(1) -\chi(1)} (\log q(\chi))^{\chi(1)-\langle \chi, \mathbf{1}_G \rangle}
\]		
and
\[
|\kappa(\chi)| \gg_{[k:\Q],|G|,\chi(1)} \epsilon(\chi) (\log eD_k)^{\mu(\widetilde{\chi}) - \mu(\chi)}  (\log q(\chi))^{\mu(\chi)-\langle \chi, \mathbf{1}_G \rangle} ,
\]
where
\begin{equation} \label{eqn:epsilon-def}
\epsilon(\chi) := \min \Big\{\Big( \frac{\log q(\psi)}{ (D_k q(\psi))^{1/2[k:\Q] } } \Big)^{\langle \chi, \psi \rangle} :   \psi \in \Psi_{K/k}(G)  \Big\} \cup \{1\}. 	
\end{equation}

\end{theorem} 
This result might seem counterintuitive since the exceptional character $\psi_{K/k}$ is unique. However, even if $\chi$  satisfies $\langle \chi, \psi_{K/k}\rangle = 0$, the character $\chi$ might still possess other components $\psi \in \Psi_{K/k}(G)$ whose $L$-functions $L(s,\psi)$ have real zeros close to $s=1$. Crucially, the notion of ``close to $s=1$'' is \textit{relative to their conductor $q(\psi)$}, which might be  smaller than any power of $D_K$ or even $\log D_K$ but the putative real zero still impacts the size of $\kappa(\chi)$.   
	
\cref{thm:cond} has  two immediate corollaries. 

\begin{corollary} \label[corollary]{cor:nonexceptional}
	If $\chi$ is a character of $G$ without trivial or quadratic components, then 
	\[
	|L(1,\chi)| \ll_{[k:\mathbb{Q}],|G|,\chi(1)} (\log eD_k)^{\widetilde{\chi}(1) -\chi(1)} (\log q(\chi))^{\chi(1)}
	\]
	and
	\[
	|L(1,\chi)|  \gg_{[k:\Q],|G|,\chi(1)}	(\log eD_k)^{\mu(\widetilde{\chi}) - \mu(\chi)} (\log q(\chi))^{\mu(\chi)}. 
	\]
	In particular, the above holds for any irreducible character $\chi$ of degree $\geq 2$. 
\end{corollary}

\begin{corollary} \label[corollary]{cor:irreducibles}
	If $\chi$ is a character of $G$   then 
	\[
	|\kappa(\chi)| \ll_{[k:\Q],|G|,\chi(1)}  (\log eD_k)^{([k:\Q]-1) \chi(1)} \prod_{\psi \neq \mathbf{1}_G} (\log q(\psi))^{\psi(1) \langle \chi, \psi\rangle}
	\]
	and
	\[
	|\kappa(\chi)| \gg_{[k:\Q],|G|,\chi(1)} \epsilon(\chi)  (\log eD_k)^{- ([k:\Q]-1) \chi(1) + ([k:\Q]-2) \langle\chi, \mathbf{1}_G\rangle} \prod_{\psi \neq \mathbf{1}_G}  (\log q(\psi))^{-\psi(1)\langle \chi, \psi \rangle},
	\]
	where $\psi$ runs over  the nontrivial irreducible characters of $G$. 
\end{corollary}	
	
We emphasize that \cref{thm:disc,thm:cond} are unconditional, and in particular do not rely on any unproven hypotheses about the zeros or poles of Artin $L$-functions, nor do they place restrictions on the characters $\chi$ to which they apply.  This is in notable contrast to Stark's work, which, while unconditional, relies on two nontrivial properties of Dedekind zeta functions: they are known to be entire away from $s=1$, and their Dirichlet series coefficients are non-negative.  The former is a technical barrier to generalization, the latter structural to the method, but both are overcome by \cref{thm:disc,thm:cond}, which we necessarily prove by different means.  See \S\ref{sec:examples} for additional novelties, and \S\ref{sec:sketch} for a discussion of our approach.

\subsection*{Organization} Section 2 provides some illustrative examples and Section 3 outlines the proof. Section 4 introduces notation for Artin $L$-functions and preliminary lemmas. Section 5 records lemmas on Landau--Siegel zeros of Artin $L$-functions. Section 6 establishes short Euler product approximations for Artin $L$-functions. Section 7 prepares three key propositions to prove our main theorem. Section 8 contains the proofs of \cref{thm:disc,thm:cond} and \cref{cor:nonexceptional,cor:irreducibles}. Section 9 establishes the sole conditional bound from \cref{prop:grh-bound}. 
	
\subsection*{Acknowledgements}

 The authors would like to thank Jesse Thorner for useful discussions and their encouragement, and Haonan Zhao for pointing out several typos in an earlier version.
PJC was supported by the National Research Foundation of Korea (NRF) grant funded by the Korea government (MSIT) (No. RS-2022-NR069491 and No. RS-2025-02262988). RJLO was partially supported by NSF grant DMS-2200760 and by the Office of the Vice Chancellor for Research at the University of Wisconsin-Madison with funding from the Wisconsin Alumni Research Foundation. AZ was partially supported by NSERC grant RGPIN-2022-04982.

%
%
%
%	SECTION 2: EXAMPLES
%
%
%
\section{Some examples} \label{sec:examples}

\subsection{Dirichlet $L$-functions} \cref{thm:cond} can be applied to any primitive Dirichlet character $\chi \pmod{q}$ of order $\ell \geq 2$. If $\ell = 2$ then it yields 
\[
\frac{1}{q^{1/2}} \ll |L(1,\chi)| \ll \log q,
\]
and if $\ell \geq 3$ (or $\ell=2$ and $L(s,\chi)$ does not have a Landau--Siegel zero) then it yields 
\begin{equation} \label{eqn:dirichlet}
	\log q 
		\gg_\ell |L(1,\chi)|
		\gg_\ell  \begin{cases}
			(\log q)^{-1} & \text{if $\ell$ is even} \\
			(\log q)^{-\cos (\frac{\pi}{\ell})} & \text{if $\ell$ is odd}.
		\end{cases}
\end{equation}
Upper bounds of the same quality are classical consequences of character sum bounds (see, e.g., \cite{Pintz,GranvilleSound} for a brief history). The lower bound for $\ell=2$ follows from the class number formula for quadratic fields, and for $\ell \geq 3$, the estimate $|L(1,\chi)| \gg (\log q)^{-1}$ is a classical consequence of zero-free regions for Dirichlet $L$-functions (see, e.g., \cite[Theorem 11.4]{MontgomeryVaughan}).  We do not claim any novelty in the improvement over this bound for odd $\ell \geq 3$ (and consider it something of a folk result, possibly dating to \cite{GranvilleSound-2007}), but we have been unable to find it in the literature.  However, we note that a strong conditional version for prime $\ell$ does appear in recent work of Darbar--David--Lalin--Lumley \cite[Proposition 1.7]{DarbarDavidLalinLumley}.

\subsection{Dedekind zeta functions}
Theorem~\ref{thm:cond} recovers Stark's effective lower bound on residues of Dedekind zeta functions.  Let $F/\mathbb{Q}$ be a number field, $K/\mathbb{Q}$ its normal closure, and let $\chi_F$ be the character of $\mathrm{Gal}(K/\mathbb{Q})$ such that $L(s,\chi_F) = \zeta_F(s) / \zeta(s)$.  Note that $\mathrm{Res}_{s=1} \zeta_F(s) = L(1,\chi_F)$, so we may apply \cref{thm:cond} to bound the residue.  We have $\widetilde{\chi}_F = \chi_F, q(\chi_F) = D_F, \chi_F(1) = [F:\Q]-1,$ and $\mu(\chi_F) = -1$. If $F$ does not contain a quadratic subfield, then $\epsilon(\chi_F) = 1$ implying
	\[
		\frac{1}{\log D_F} 
			\ll_{[F:\mathbb{Q}]} \mathrm{Res}_{s=1} \zeta_F(s) 
			\ll_{[F:\mathbb{Q}]} (\log D_F)^{[F:\mathbb{Q}]-1}. 
	\]
On the other hand, if $F$ does contain a quadratic subfield, then $\epsilon(\chi_F) \gg_{[F:\Q]} D_F^{-1/[F:\Q]} \log D_F$ by the conductor-discriminant formula (cf. \cref{lem:conductor-bound}) so we obtain instead
	\[
		\frac{1}{D_F^{1/[F:\mathbb{Q}]}}
			\ll_{[F:\mathbb{Q}]} \mathrm{Res}_{s=1} \zeta_F(s)
			\ll_{[F:\mathbb{Q}]} (\log D_F)^{[F:\mathbb{Q}]-1}.
	\]
These match Landau and Stark's effective bounds on the residue.
	
\subsection{Choice of base field} \label{subsec:example-basefield}
By taking a suitable induction, any Artin $L$-function $L(s,\chi)$ over $k$ may be regarded as an Artin $L$-function $L(s,\chi^*)$ over some subfield $k^* \subseteq k$, say $k^* = \Q$.  Applying Theorem~\ref{thm:cond} then yields multiple bounds on $L(1,\chi) = L(1,\chi^*)$ depending on the choice of base field $k$ or $k^*$. Thus, by carefully extracting the dependence on the base field, Theorem~\ref{thm:cond} reveals an interesting, and possibly new, phenomenon.  We highlight this phenomenon in the simplest interesting example.

Let $F/\mathbb{Q}$ be a non-Galois cubic field of discriminant $df^2$, and let $K/\mathbb{Q}$ denote its normal closure.  Let $\chi_F$ be the character of $\mathrm{Gal}(K/\mathbb{Q}) \simeq S_3$ so that $L(s,\chi_F) = \zeta_F(s) / \zeta(s)$.  Thus, $\chi_F$ is the character of the $2$-dimensional standard representation of $S_3$, so we have $\chi_F(1) = 2$ and $\mu(\chi_F) = -1$, and $L(s,\chi_F)$ does not have a Landau--Siegel zero since it has no quadratic component.  We thus obtain from Theorem~\ref{thm:cond} that
\begin{equation} \label{eqn:first-S3}
	\frac{1}{\log( d f^2)} \ll L(1,\chi_F) \ll (\log df^2)^2.
\end{equation}
On the other hand, the character $\chi_F$ is monomial.  Let $\psi_F$ be the nontrivial character of $\mathrm{Gal}(K/\mathbb{Q}(\sqrt{d})) \simeq C_3$ whose induction to $\mathrm{Gal}(K/\Q)$ is precisely $\chi_F$, so that 
\[
L(s,\psi_F) = L(s,\chi_F).
\]
As before, $L(s,\psi_F)$ does not have a Landau--Siegel zero since $\psi_F$ is 1-dimensional of order 3, but now we have that $\psi_F(1) = 1$ and $\mu(\psi_F) = -\frac{1}{2}$.  Hence Theorem~\ref{thm:cond} applied to $L(s,\psi_F)$, viewed as an $L$-function over $k = \mathbb{Q}(\sqrt{d})$, yields  		\begin{equation} \label{eqn:second-S3}
	\frac{1}{(\log d)^{1/2}(\log f)^{1/2}} \ll L(1,\psi_F) \ll (\log d)(\log f).
\end{equation}
But since $L(1,\psi_F) = L(1,\chi_F)$, observe \eqref{eqn:second-S3} yields a uniformly better bound than \eqref{eqn:first-S3}.  We leave open the questions of how general this example may be and how to choose the optimal base field for any given Artin $L$-function.   	

\subsection{Choice of decomposition} We can bound $\kappa(\chi)$ by decomposing a character $\chi$ into components and then applying \cref{thm:cond} to each component. This strategy is often used in character theory to reduce analysis to irreducible characters. While this approach would succeed to establish the upper bound in \cref{thm:cond}, we do not see how to deduce the lower bound in this manner.  The essential issue is that the degree map $\chi \mapsto \chi(1)$ is linear whereas the map $\chi \mapsto \mu(\chi)$ is sublinear, i.e. $(\chi+\chi')(1) = \chi(1) + \chi'(1)$ whereas $\mu(\chi+\chi') \geq \mu(\chi) +\mu(\chi')$. We highlight this distinction with two simple examples. 

Let $K/\Q$ be a Galois extension with $G = \mathrm{Gal}(K/\Q)$. Let $\psi$ and $\psi'$ be two distinct characters which have no trivial or quadratic components. Define $\chi = \psi + \psi'$, so 
\[
L(s,\chi) = L(s,\psi) L(s,\psi') \quad \text{ and } \quad q(\chi) = q(\psi) q(\psi'). 
\]
We can estimate this quantity at $s=1$ in several ways.  \cref{thm:cond} applied to $L(1,\chi)$ gives 
\begin{equation} \label{eqn:decomp1}
(\log q(\psi) + \log q(\psi'))^{\mu(\psi + \psi')} \ll_{|G|, \chi(1)} L(1,\chi) \ll_{|G|, \chi(1)} (\log q(\psi) + \log q(\psi'))^{ (\psi+\psi')(1)}. 
\end{equation}
\cref{thm:cond} applied to $L(1,\psi)$ and $L(1,\psi')$ gives
\begin{equation} \label{eqn:decomp2}
(\log q(\psi))^{\mu(\psi)} (\log q(\psi'))^{\mu(\psi')} \ll_{|G|, \chi(1)} L(1,\chi) \ll_{|G|, \chi(1)} (\log q(\psi))^{\psi(1)} (\log q(\psi'))^{\psi'(1)}. 
\end{equation} 
The optimal upper bound is provided by the second estimate since $(\psi+\psi')(1) = \psi(1) + \psi'(1)$; this pattern holds more generally as described by the remark at the end of \S\ref{sec:proofs-main}. 

The optimal lower bound  is less clear. For example, if $G \simeq D_5$ and $\psi$ and $\psi'$ are the two  distinct irreducible characters of degree $2$, then $\mu(\psi+\psi') = -1$ and $\mu(\psi) = \mu(\psi') = -\frac{1+\sqrt{5}}{2}$,  so   \eqref{eqn:decomp1} is better than \eqref{eqn:decomp2}.  More generally, if $\psi$ and $\psi'$ are such that $\log q(\psi) \asymp_{|G|,\psi(1),\psi^\prime(1)} \log q(\psi')$ (which is always the case if $\psi$ and $\psi^\prime$ are faithful; cf. \cref{lem:conductor-bound}), then by sublinearity $\mu(\psi+\psi') \geq \mu(\psi) + \mu(\psi')$, so \eqref{eqn:decomp1} is always better in this case. 

On the other hand, the second lower bound \eqref{eqn:decomp2} can be better when one of $\psi$ or $\psi'$ is not faithful. Concretely, let $G \simeq \mathbb{F}_5 \rtimes \mathbb{F}_5^\times$ be the Frobenius group of order $20$. Let $\psi = \psi_4 + \overline{\psi}_4$ where $\psi_4$ and $\overline{\psi}_4$ are the characters of degree $1$ and order $4$ coming from the quotient map $G \to \mathbb{F}_5^{\times} \simeq C_4$. Let $\psi'$ be the unique irreducible character of degree $4$, which is faithful.  Then $\mu(\psi) = -2$, $\mu(\psi^\prime)=-1$, and $\mu(\psi + \psi^\prime) = -2$. However, the character $\psi$ is not faithful by construction, so it can happen that $\log q(\psi) = o(\log q(\psi^\prime))$, and certainly $\log q(\psi) \ll \log q(\psi^\prime)$ always.  Then \eqref{eqn:decomp1} gives $L(1,\psi+\psi^\prime) \gg (\log q(\psi^\prime))^{-2}$, while \eqref{eqn:decomp2} gives $L(1,\psi+\psi^\prime) \gg (\log q(\psi))^{-2} (\log q(\psi^\prime))^{-1}$.  This second bound is evidently an improvement over the first if $\log q(\psi) \ll (\log q(\psi'))^{1/2}$.  We leave open questions of how to choose the optimal decomposition for any given Artin $L$-function. 

\subsection{Conditional bounds} Finally, we showcase one conditional result on the generalized Riemann hypothesis, whose short proof appears in \cref{sec:grh}. 
 
\begin{proposition} \label[proposition]{prop:grh-bound}
Let $K/k$ be a Galois extension of number fields with Galois group $G$. Let $\chi$ be any character of $G$. If the generalized Riemann hypothesis holds for $\zeta_K(s)$, then
\[
|\kappa(\chi)| \ll_{[k:\Q],|G|,\chi(1)} (\log\log eD_k)^{\widetilde{\chi}(1)-\chi(1)} (\log\log q(\chi))^{\chi(1)-\langle \chi, \mathbf{1}_G \rangle},
\]
and 
\[
|\kappa(\chi)| \gg_{[k:\Q],|G|,\chi(1)} (\log\log eD_k)^{\widetilde{\mu}(\chi)-\mu(\chi)} (\log\log q(\chi))^{\mu(\chi)-\langle \chi, \mathbf{1}_G \rangle}.
\]
\end{proposition}		
Observe that \cref{prop:grh-bound} is the same qualitative shape as \cref{thm:cond} with ``$\log$'' replaced by ``$\log\log$'' everywhere, and without any exceptional zeros. This same gap between unconditional and conditional bounds is well known for Dirichlet $L$-functions (see, e.g., \cite{GranvilleSound}), and our bounds  match (up to implied constants) the best known for Dirichlet $L$-functions due to Littlewood \cite{Littlewood} for the upper bound and Darbar--David--Lalin--Lumley \cite[Proposition 1.7]{DarbarDavidLalinLumley} for the lower bound.  We therefore  suspect \cref{prop:grh-bound} indicates how \cref{thm:cond} is near the limit of existing unconditional methods.

%
%
%
%	SECTION 3: METHOD OVERVIEW
%
%
%

\section{Method overview} \label{sec:sketch} 

Unlike Stark and Landau's approach for the Dedekind zeta function, we prove \cref{thm:disc,thm:cond} by approximating   $\kappa(\chi)$ with a short Euler product of $L(s,\chi)$ truncated at length $T$. This result may be of independent interest, so we record it here. 
	
\begin{theorem}  \label{thm:short-euler} There exists absolute positive constants $c_3$ and $c_4$ such that the following holds: 
	Let $K/k$ be a Galois extension of number fields with Galois group $G$.  For any character $\chi$ of $G$ and any $T \geq 3([K:\Q]^{[K:\Q]} D_K)^{c_3}$,    
		\small 
		\[
		\kappa(\chi) = 
			 \frac{\widetilde{\eta}(\chi,T)}{(e^\gamma\log T)^{\langle \chi, \mathbf{1}_G\rangle} } \Big( \prod_{\N\kp \leq T} L_{\kp}(1,\chi) \Big)  \bigg\{ 1 +  O\bigg(\exp\Big(\frac{-c_4 \log T}{\log([K:\Q]^{[K:\Q]} D_K) + ([K:\Q]\log T)^{1/2}}\Big)  \bigg) \bigg\}^{\chi(1)},
		\]
		\normalsize
		 where, denoting $\psi_{K/k}$ as the exceptional character of $K/k$ (if it exists), we define 
		 \[
		\widetilde{\eta}(\chi,T) 
	= \begin{cases}
	 		\exp\Big( -\displaystyle \int_T^{\infty} \frac{\langle \chi, \psi_{K/k}\rangle}{t^{2-\beta_K} \log t} dt \Big)   & \text{if $\zeta_K(s)$ has a real zero $\beta_K > 1 - \frac{1}{4 \log D_K}$,} \\[10pt]
	 		1 & \text{otherwise.}
	 \end{cases} 			
	\]
	Here $\gamma =  0.5772\dots$ is the Euler--Mascheroni constant over $\Q$. 
\end{theorem}

The proof appears in \S\ref{sec:short-euler} and must deal with three main obstacles. First, we cannot assume holomorphy of $L(s,\chi)$. Second, the length $T$ must be short enough, roughly $\log T \approx \log D_K$, to obtain bounds relative to the conductor $q(\chi)$ since $\log q(\chi) \approx \log D_K$ for faithful characters $\chi$ (cf. \cref{lem:conductor-bound}); note the faithful case is sufficient after passing to a subextension of $K/k$. Third, the approximation must  account for the exceptional zero $\beta_K$ of $\zeta_K(s)$, if it exists. All three challenges are overcome by invoking a recent uniform version  of the Chebotarev density theorem due to Thorner and Zaman \cite{TZ} (cf. \cref{thm:chebotarev}).

Notice \cref{thm:short-euler} gives an asymptotic short Euler product approximation for $L(s,\chi)$ but, viewing $L(s,\chi) = \prod_{\psi} L(s,\psi)^{\langle\chi,\psi \rangle}$ as a product over $\psi \in \mathrm{Irr}(G)$, the theorem also approximates every $L(s,\psi)$ with a \textit{single} fixed truncation $T$ satisfying $\log T \approx \log D_K$. This choice is acceptable for the proof of \cref{thm:disc}, but not for \cref{thm:cond} as it requires sensitivity to the conductor sizes. The  analytic conductors $q(\psi)$ might vary substantially over $\psi$, so $\log q(\psi)$ might be  small compared to $\log T$ for some $\psi \in \mathrm{Irr}(G)$. If the truncation parameter $T$ is too long for these $\psi$, then our approximation of $L(s,\psi)$ will be too crude.  We  introduce varied truncation parameters $T(\psi)$   to better estimate each  $L(1,\psi)^{\langle \chi, \psi\rangle}$ with a ``customized'' short Euler product. As a tradeoff, we only obtain the order of magnitude and we must  account for the possibility of \textit{multiple} potentially exceptional characters $\psi \in \Psi_{K/k}(G)$ from \eqref{eqn:exceptional-quadratics} and their real zeros $\beta_{\psi} > 1 - \frac{1}{4 \log q(\psi)}$. The following proposition, proved at the end of \S\ref{sec:key-props}, illustrates this outcome in a precise manner.

\begin{proposition} \label[proposition]{prop:residue-optimal}
	Let $K/k$ be a Galois extension of number fields with Galois group $G$. For any character $\chi$ of $G$,     
		\[
		\kappa(\chi) \asymp_{[k:\Q],|G|,\chi(1)} 
			 \frac{1}{(\log (eD_k))^{\langle \chi, \mathbf{1}_G\rangle} } \prod_{\psi} \Big( \eta(\psi)  \prod_{\N\kp \leq q(\psi) } L_{\kp}(1,\psi) \Big)^{\langle \chi, \psi\rangle}  ,
		\]
		 where the product runs over irreducible characters $\psi$ of $G$ and  
\begin{equation} \label{eqn:eta-def}
\eta(\psi) := 
\begin{cases}
	(1-\beta_{\psi}) \log q(\psi)  & \text{if $\psi \in \Psi_{K/k}(G)$,} \\[3pt]
	1 & \text{otherwise.} 
 \end{cases}
\end{equation}
\end{proposition}
Although we utilize a slightly more flexible version (\cref{prop:residue}) to prove \cref{thm:cond}, this proposition already demonstrates how to more carefully extract the behaviour of primes along varying scales of conductors, and hence obtain good dependence on the base field. 

The final step is to naively bound the remaining two products:
\[
P := \prod_{\psi} \prod_{\N\kp \leq q(\psi)} L_{\kp}(1,\psi)^{\langle \chi, \psi \rangle} \quad \text{ and } \quad E := \prod_{\psi} \eta(\psi)^{\langle \chi, \psi \rangle }. 
\]
We carefully estimate $P$ (\cref{prop:small-primes}) by appealing to a sharp upper bound form of Mertens' formula over number fields (\cref{lem:mertens}), the classical asymptotic form of Mertens formula over $\Q$, and a delicate telescoping decomposition over ranges of primes. We carefully estimate $E$ (\cref{prop:eta}) by appealing to Stark's results on exceptional zeros and  a strong form of zero repulsion known as the Deuring--Heilbronn phenomenon described in \S\ref{sec:landau-siegel}.

%
%
%
%	SECTION 4: PRELIMINARIES
%
%
%
		
\section{Preliminaries on Artin $L$-functions} \label{sec:prelim}

In this section, we provide a streamlined overview of Artin $L$-functions as we shall use them, in particular our notations and conventions.  We also prove two elementary but perhaps not entirely obvious facts that underlie the proofs of our main theorems.  We do not aim for an exhaustive discussion aimed at readers wholly unfamiliar with Artin $L$-functions; we refer such readers instead to excellent sources such as \cite{MurtyMurty,Neukirch,Lombardo}.

Let $K/k$ be a Galois extension of number fields and let $V$ be a finite dimensional complex representation of $G := \mathrm{Gal}(K/k)$.  Let $\chi$ be the corresponding character, i.e. $\chi(g) = \mathrm{tr}(g | V)$.  The Artin $L$-function $L(s,\chi)$ is then defined as an Euler product
	\[
		L(s,\chi) := \prod_{\mathfrak{p}} L_\mathfrak{p}(s,\chi)
	\]
running over prime ideals $\mathfrak{p}$ of $k$, where the local factors $L_\mathfrak{p}(s,\chi)$ are defined as follows.  Let $D_\mathfrak{p}$ and $I_\mathfrak{p}$ be the decomposition and inertia groups of a prime lying over $\mathfrak{p}$, and let $\sigma_\mathfrak{p} \in D_\mathfrak{p}$ map to Frobenius under the canonical isomorphism between $D_\mathfrak{p}/I_\mathfrak{p}$ and the Galois group of the residue fields.  Finally, write $\mathrm{N}\mathfrak{a} := | \mathcal{O}_k/\mathfrak{a}|$ for the norm of any non-zero integral ideal $\mathfrak{a}$ of $k$.  The local Euler factor is then defined by
	\begin{equation} \label{eqn:local-factor}
	\begin{aligned}
		L_\mathfrak{p}(s,\chi)
			&:= \det\left(1 - (\mathrm N \mathfrak{p})^{-s} \sigma_\mathfrak{p} | V^{I_\mathfrak{p}} \right)^{-1} = \prod_{j=1}^{\chi(1)} \Big( 1 - \frac{\alpha_{j,\kp}}{\N\kp^s} \Big)^{-1},
	\end{aligned}
	\end{equation}
where $V^{I_\mathfrak{p}}$ denotes the subspace of $V$ fixed by the action of $I_\mathfrak{p}$, and $\{ \alpha_{j,\kp} \}_j$ are complex numbers satisfying $|\alpha_{j,\kp}| \leq 1$.  In particular, if $\mathfrak{p}$ is unramified in $K$, then $I_\mathfrak{p} = 1$ and $V^{I_\mathfrak{p}} = V$, so $\mathrm{tr}(\sigma_\mathfrak{p} | V^{I_\mathfrak{p}}) = \sum_{j=1}^{\chi(1)} \alpha_{j,\kp} = \chi(\sigma_\mathfrak{p})$.  Motivated by this, we define 
\begin{equation} \label{eqn:chi-p-def}
\chi(\mathfrak{p}) := \mathrm{tr}(\sigma_\mathfrak{p} | V^{I_\mathfrak{p}}),		
\end{equation}
including in the case that $\mathfrak{p}$ is ramified and $\chi(\mathfrak{p})$ is not given directly as a character value. Uniform bounds for the local factors $L_{\kp}(1,\chi)$ in terms of $\chi(\kp)$ will be crucial to our arguments, so we record a simple lemma. 

\begin{lemma}\label[lemma]{lem:local-bound} Let $K/k$ be a Galois extension of number fields with Galois group $G$. Let $\chi$ be any character of $G$.  For any prime ideal $\mathfrak{p}$ of $k$, we have 
	\[
		\mu(\chi) \leq \Re(\chi(\mathfrak{p})) \leq \chi(1) 
	\] 
	and 
	\[
	\Big| \log L_{\kp}(1,\chi) - \frac{\chi(\kp)}{\N\kp} \Big| \leq \frac{2 \chi(1)}{\N\kp^2}. 
	\]
	Therefore, 
	\[
		  \exp\Big(\frac{\mu(\chi)}{\N\kp} - \frac{2 \chi(1)}{\N\kp^2} \Big) \leq |L_{\kp}(1,\chi)| \leq \exp\Big( \frac{\chi(1)}{\N\kp} + \frac{2 \chi(1)}{\N\kp^2} \Big).  
	\]
\end{lemma}
\begin{proof}  
The first claim for $\Re(\chi(\kp))$ follows immediately if $\mathfrak{p}$ is unramified in $K$, since $\chi(\mathfrak{p}) = \chi(\sigma_\mathfrak{p})$.  If $\mathfrak{p}$ is ramified, then we have instead $\chi(\mathfrak{p}) = \mathrm{tr}(\sigma_\mathfrak{p} \mid V^{I_\mathfrak{p}})$.  This is sufficient for the upper bound, since we have $\Re(\chi(\mathfrak{p})) \leq \dim V^{I_\mathfrak{p}} \leq \chi(1)$.  It suffices also for the lower bound on noting that
		\[
			\mathrm{tr}(\sigma_\mathfrak{p} \mid V^{I_\mathfrak{p}})
				= \frac{1}{|I_\mathfrak{p}|} \sum_{g \in \sigma_\mathfrak{p} I_\mathfrak{p}} \chi(g),
		\]
	which evidently has real part $\geq \mu(\chi)$. The second claim for $L_{\kp}(1,\chi)$ follows after taking the logarithm of \eqref{eqn:local-factor} at $s=1$. Indeed, by a Taylor expansion of the log, we see that 
	\[
 \log L_{\kp}(1,\chi) =  - \sum_{j=1}^{\chi(1)} \log \Big(1-\frac{\alpha_{j,\kp}}{\N\kp}\Big) = \sum_{m=1}^{\infty} \frac{1}{\N\kp^m} \Big(\sum_{j=1}^{\chi(1)}  \alpha_{j,\kp}^m\Big). 
	\]
	The $m=1$ term corresponds to $\chi(\kp)/\N\kp$ by definition \eqref{eqn:chi-p-def}, and the $m \geq 2$ terms can be bounded trivially since $|\alpha_{j,\kp}| \leq 1$ for all $j$, so
	\[
	\Big|  \log L_{\kp}(1,\chi) - \frac{\chi(\kp)}{\N\kp} \Big| \leq \chi(1) \sum_{m=2}^{\infty} \frac{1}{\N\kp^m} = \frac{\chi(1)}{\N\kp^2} \cdot \frac{1}{1-1/\N\kp} \leq \frac{2 \chi(1)}{\N\kp^2}, 
	\] 
	as desired. The third claim is an immediate consequence of the first two. 
\end{proof}

We will make extensive use of the linearity of Artin $L$-functions,
	\[
		L(s,a_1\chi_1 + \dots + a_r \chi_r)
			= L(s,\chi_1)^{a_1} \dots L(s,\chi_r)^{a_r},
	\]
and we find it convenient to state many of our results in terms of the usual inner product on characters,
	\[
		\langle \chi_1, \chi_2 \rangle
			:= \frac{1}{|G|} \sum_{g \in G} \chi_1(g) \overline{\chi_2(g)}.
	\]
It is known that every Artin $L$-function $L(s,\chi)$ is meromorphic on $\mathbb{C}$, analytic and non-vanishing on $\Re(s) \geq 1$ except possibly at $s=1$, where $L(s,\chi)$ has a pole of order $\langle \chi, \mathbf{1}_G \rangle$, where $\mathbf{1}_G$ is the trivial character of $G$.  It follows from these considerations that if $\langle \chi, \mathbf{1}_G\rangle = 0$, then $L(1,\chi)$ is defined, non-zero, and given by the value at $s=1$ of its defining Euler product:
	\begin{equation} \label{eqn:prime-power-removal}
		L(1,\chi)
			= \prod_\mathfrak{p} L_\mathfrak{p}(1,\chi)		
	\end{equation}

Finally, define the analytic conductor 
\[
q(\chi) = |\mathrm{Disc}(k)|^{\chi(1)} \mathrm{N}\mathfrak{f}_\chi,
\]
where $\mathfrak{f}_\chi$ is the Artin conductor of $\chi$ as defined in \cite{Neukirch} for example. If $\widetilde{K}$ denotes the normal closure of $K$ over $\Q$, then by means of the quotient map $\mathrm{Gal}(\widetilde{K}/k) \to \mathrm{Gal}(K/k)$, we may regard $\chi$ also as a character of $\mathrm{Gal}(\widetilde{K}/k)$ and hence consider its induction $\widetilde{\chi}$ to $\mathrm{Gal}(\widetilde{K}/\Q)$.  So doing, we have 
	\[
	q(\widetilde{\chi}) = q(\chi), \quad \widetilde{\chi}(1) = [k:\Q] \chi(1), \quad \text{ and } \quad L(s,\chi) = L(s,\widetilde{\chi}). 
	\]
We more simply refer to $\widetilde{\chi}$ as the induction of $\chi$ to $\mathbb{Q}$.

These conductors can be bounded in terms of the discriminant of $K$. 
\begin{lemma} \label[lemma]{lem:conductor-bound}
	Let $K/k$ be a Galois extension of number fields with Galois group $G$. Let $\chi$ be any character of $G$. Then $q(\chi) \leq D_K^{2\chi(1)/|G|}$ and, if $\chi$ is faithful, then $q(\chi) \geq D_K^{1/|G|}$.
\end{lemma}
\begin{proof}
	Let $\mathfrak{p}$ be a prime ideal of $k$ and let $\chi$ be any character of $G$.  By \cite[Corollary~VI.2.1]{Serre-LocalFields}, we have
		\begin{equation}
			\label{eqn:artin-conductor-valuation}
			v_\mathfrak{p}(\mathfrak{f}_\chi)
				= \sum_{i\geq 0} \frac{|G_i|}{|G_0|}\left(\chi(1) - \frac{1}{|G_i|} \sum_{g \in G_i} \chi(g)\right),
		\end{equation}
	where $G_i$ for $i\geq 0$ denote the lower ramification groups at a prime of $K$ lying over $\mathfrak{p}$.  Temporarily taking $\chi=\mathrm{reg}$ to be the character of the regular representation of $G$, since $\mathfrak{f}_\mathrm{reg} = \mathfrak{D}_{K/k}$, the relative discriminant of $K/k$, we find in particular that
		\begin{equation}
			\label{eqn:relative-discriminant-valuation}
			v_\mathfrak{p}(\mathfrak{D}_{K/k})
				= \sum_{i \geq 0} \frac{|G_i|}{|G_0|}\left( |G| - \frac{|G|}{|G_i|}\right)
				= |G| \cdot \sum_{i \geq 0} \frac{|G_i|-1}{|G_0|}.
		\end{equation}
	We now essentially compare the contributions to \eqref{eqn:artin-conductor-valuation} and \eqref{eqn:relative-discriminant-valuation} from a fixed $i$.  Noting that
		\[
			\chi(1) - \frac{1}{|G_i|}\sum_{g \in G_i} \chi(g)
				= \chi(1) - \langle \mathbf{1}_{G_i}, \chi\vert_{G_i} \rangle
				\leq \chi(1),
		\]
	we first see that
		\[
			v_\mathfrak{p}(\mathfrak{f}_\chi)
				\leq \frac{\chi(1)}{|G|} \cdot \max\left\{ \frac{|G_i|}{|G_i|-1} : G_i \ne 1\right\} \cdot v_\mathfrak{p}(\mathfrak{D}_{K/k})
				\leq \frac{2\chi(1)}{|G|} \cdot v_\mathfrak{p}(\mathfrak{D}_{K/k}),
		\]
	which, since $D_K = D_k^{|G|} \N\mathfrak{D}_{K/k}$, implies
		\[
			q(\chi)
				= D_k^{\chi(1)} \N\mathfrak{f}_\chi
				\leq D_k^{\chi(1)} \N\mathfrak{D}_{K/k}^{\frac{2\chi(1)}{|G|}}
				= D_k^{-\chi(1)} D_K^{\frac{2\chi(1)}{|G|}}.
		\]
	This gives the first inequality.
	
	For the second, we observe that unless $G_i \leq \ker \chi$, then the expression
		\[
			\chi(1) - \frac{1}{|G_i|} \sum_{g \in G_i} \chi(g)
		\]
	must be positive, and hence at least $1$ since it is an integer.  If now $\chi$ is assumed to be faithful, then by definition $\ker \chi = 1$, which implies that
		\[
			v_\mathfrak{p}(\mathfrak{f}_\chi)
				\geq \sum_{\substack{i \geq 0 \\ : G_i \ne 1}} \frac{|G_i|}{|G_0|}
				\geq \frac{v_\mathfrak{p}(\mathfrak{D}_{K/k})}{|G|} \min\left\{ \frac{|G_i|}{|G_i|-1} : G_i \ne 1\right\}
				\geq \frac{v_\mathfrak{p}(\mathfrak{D}_{K/k})}{|G|}. 
		\]
	Similar to the above, this now implies
		\[
			q(\chi)
				= D_k^{\chi(1)} \N\mathfrak{f}_\chi
				\geq D_k^{\chi(1)} \N\mathfrak{D}_{K/k}^{1/|G|}
				= D_k^{\chi(1)-1} D_K^{1/|G|}.
		\]
	This completes the proof. 
\end{proof}

%
%
%
%	SECTION 5: EXCEPTIONAL ZEROS
%
%
%
 
\section{Exceptional zeros and potentially exceptional characters}
	\label{sec:landau-siegel}
	
We record some facts and properties related to Landau--Siegel zeros of Artin $L$-functions, much of which are essentially due to Stark. Recall   for a Galois extension $K/k$ of number fields with Galois group $G$, we have precisely defined the exceptional zero $\beta_K$, the exceptional character $\psi_{K/k}$ associated to $K/k$, and the set of potentially  exceptional characters $\Psi_{K/k}(G)$ in Section~\ref{sec:intro}. Since Stark formulated his theorems purely in terms of number fields, we rephrase his major innovations in terms of Artin $L$-functions.  

\begin{lemma}[Stark, Ahn--Kwon] \label[lemma]{lem:stark-zfr}
	Let $K/k$ be a Galois extension of number fields with Galois group $G$. For any irreducible character $\psi \in \mathrm{Irr}(G)$ satisfying $\psi^2 = \mathbf{1}_G$, the 1-dimensional Artin $L$-function $L(s,\psi)$ has at most one zero in the region
	\[
	\Re(s) > 1 - \frac{1}{4 \log q(\psi)}, \qquad |\Im(s)| \leq \frac{1}{4\log q(\psi)}. 
	\]
	If this zero $\beta_{\psi}$ exists, it is real and simple. 
\end{lemma}
\begin{proof}
	The   trivial and   quadratic cases were  respectively proved by Stark \cite[Lemma 3]{Stark} and Ahn--Kwon \cite[Corollary 1]{AhnKwon}. 
\end{proof}

Building on a result of Heilbronn \cite{Heilbronn}, Stark established that the source of the exceptional zero $\beta_K$ for the Dedekind zeta function $\zeta_K(s)$ is precisely a unique character from the subset $\Psi_{K/k}(G)$ of   exceptional characters. 

\begin{lemma}[Stark] \label[lemma]{lem:stark-artin}
	Let $K/k$ be a Galois extension of number fields with Galois group $G$.  If $\zeta_K(s)$ has an exceptional zero $\beta_K$, then there is a unique character $\psi_{K/k}  \in \Psi_{K/k}(G)$  such that 
	\[
	\mathop{\mathrm{ord}}_{s=\beta_K} L(s,\chi) = \langle \chi, \psi_{K/k} \rangle
	\]
	for any character $\chi$ of $G$. In particular,   every $L(s,\chi)$ is holomorphic at $s=\beta_K$, and $L(\beta_K,\chi) = 0$ if and only if $\langle \chi, \psi_{K/k} \rangle \geq 1$.  
\end{lemma}
\begin{proof}
	The statement for irreducible characters $\chi$ follows from the proof of \cite[Theorem 3]{Stark} and the observation  $\beta_K > 1 - \frac{1}{4\log D_K} \geq 1 - \frac{1}{4 \log q(\psi)}$ for any irreducible character $\psi$ of $G$.  This implies the result for general characters $\chi$ by linearity.
\end{proof}

The combination of \cref{lem:stark-zfr,lem:stark-artin} motivates our definition for the set of potentially exceptional characters $\Psi_{K/k}(G)$ in \eqref{eqn:exceptional-quadratics}.  This  reduction allowed Stark to establish an effective lower bound for $1-\beta_K$. We do the same with $1-\beta_{\psi}$ for every character $\psi \in \Psi_{K/k}(G)$.

\begin{lemma}[Stark]\label[lemma]{lem:stark-effective}
	Let $K/k$ be a Galois extension of number fields with Galois group $G$. If  $\psi \in \Psi_{K/k}(G)$ is a potentially exceptional character with real zero $\beta_{\psi}$, then
		\[
			1 - \beta_{\psi}
				\gg_{[k:\Q]} \frac{1}{(D_kq(\psi))^{1/2[k:\Q]}}.  
		\]
\end{lemma}
\begin{proof} 
	Set $\beta = \beta_{\psi}$. If $\psi$ is trivial, the claim agrees with Stark's result  appearing below his equation (27) in \cite{Stark}. Hence, we assume that $\psi$ is quadratic. Let $F$ be the quadratic extension of $k$ associated with $\psi$, so $\beta$ is a real zero of $\zeta_F(s)$ and $D_F \leq q(\psi)^2$.  Without loss of generality, we may assume 
	\[
	1-\beta  \leq \frac{1}{(2 [k:\Q])! \cdot 8 \log q(\psi) } \leq \frac{1}{[F:\Q]! \cdot 4\log D_F},
	\]
	because otherwise the inequality holds trivially. Therefore, by \cite[Lemma 8]{Stark}, there exists a quadratic field $\Q(\sqrt{d})$  with fundamental discriminant $d$ such that $\Q(\sqrt{d})$ is contained in $F$ and $\zeta_{\Q(\sqrt{d})}(\beta) = 0$.  Siegel's classical lower bound implies $1-\beta  \gg |d|^{-1/2}$. Since $\psi$ is quadratic, $d^{[k:\mathbb{Q}]}$ must divide $D_F = D_k q(\psi)$, and the result follows.
 \end{proof}
	
There may be \textit{multiple} potentially exceptional characters in $\Psi_{K/k}(G)$ which, necessarily by \cref{lem:stark-artin}, will be associated to Artin $L$-functions of conductors with drastically different sizes. If this occurs, the real zeros must repel each other. This effect can be quantified in a weak form using \cref{lem:stark-zfr} and in a strong form, commonly known as Deuring--Heilbronn phemonenon, using a result of Lagarias--Montgomery--Odlyzko \cite{LMO}. 

\begin{lemma} \label[lemma]{lem:deuring-heilbronn} 
	There exists absolute positive constants $c_1$ and $c_2$ such that the following holds. Let $K/k$ be a Galois extension of number fields with Galois group $G$. For each $i \in \{1,2\}$, let $\psi_i \in \Psi_{K/k}(G)$ be a potentially exceptional character  with real zero  $\beta_i := \beta_{\psi_i} > 1 - \frac{1}{4 \log q(\psi_i)}$. \\
	If $\psi_1 \neq \psi_2$ then 
	\[
	\min\{ \beta_1,\beta_2 \} \leq 1 - \frac{1}{12 \log( q(\psi_1) q(\psi_2) )}
	\]
	and, moreover,
	\[
	\min\{\beta_1,\beta_2\} \leq 1 - \frac{\log\Big( \dfrac{c_2}{(1-\max\{\beta_1,\beta_2\}) \log( q(\psi_1) q(\psi_2) ) } \Big) }{c_1 \log( q(\psi_1) q(\psi_2) )}.  
	\]
\end{lemma}
\begin{proof} 
	For $i \in \{1,2\}$, let $F_i$ be a quadratic or trivial extension of $k$ with associated character $\psi_i$. We only consider the case when both $\psi_1$ and $\psi_2$ are quadratic; the other cases are very similar. Thus, the compositum $F = F_1 F_2$ is a biquadratic field over $k$, so
	\[
	\zeta_F(s) = \zeta_k(s) L(s,\psi_1) L(s,\psi_2) L(s, \psi_1 \otimes \psi_2)
	\]
	and hence the Dedekind zeta function $\zeta_F(s)$ has  two real zeros $\beta_1$ and $\beta_2$, counted with multiplicity.		 Since $\psi_1$ and $\psi_2$ are 1-dimensional characters, we have that $q(\psi_1 \otimes \psi_2) \leq q(\psi_1) q(\psi_2)$ and also trivially $D_k \leq q(\psi_1) q(\psi_2)$, which implies
	 \begin{equation} \label{eqn:cond-disc-biquadratic}
	q(\psi_1) q(\psi_2) \leq  D_F = D_k q(\psi_1) q(\psi_2) q(\psi_1 \otimes \psi_2) \leq (q(\psi_1) q(\psi_2))^3.
	 \end{equation}
	 By Stark's zero free region for $\zeta_F(s)$, i.e. \cref{lem:stark-zfr}, it follows that 
	 \[
	 \min\{\beta_1,\beta_2\} \leq 1 - \frac{1}{4 \log D_F} \leq 1 - \frac{1}{12 \log ( q(\psi_1) q(\psi_2))}.
	 \]
	This establishes the first bound. 
	
	For the second bound, assume without loss of generality that $\beta_2 \leq \beta_1$. We claim that we may also assume that $\beta_1 > 1 - 1/(12 \log q(\psi_1) q(\psi_2))$ because otherwise 
	\[
	\beta_2 \leq \beta_1 < 1 \leq 1 - \frac{\log\Big( \dfrac{1/12}{(1-\beta_1) \log ( q(\psi_1) q(\psi_2) )} \Big) } {12 \log q(\psi_1) q(\psi_2) } 
	\]
	in which case the lemma holds with $c_1 = 12$ and $c_2 = 1/12$. By our assumption and \eqref{eqn:cond-disc-biquadratic}, it follows that $\beta_1 > 1-1/(4 \log D_F)$ so $\beta_1$ is an exceptional zero of $\zeta_F(s)$. The required upper bound for $\beta_2$ now follows from the Deuring--Heilbronn phemonenon established by Lagarias--Montgomery--Odlyzko \cite[Theorem 5.1]{LMO} applied to $\zeta_F(s)$. 
\end{proof}

%
%
%
%	SECTION 6: SHORT EULER PRODUCTS
%
%
%

\section{Short Euler product approximations and the proof of Theorem 3.1} \label{sec:short-euler}

In this section, we estimate  Euler products and approximate residues of Artin $L$-functions by short Euler products to establish \cref{thm:short-euler}. A crucial ingredient is a uniform version of the Chebotarev density theorem due to Thorner--Zaman \cite{TZ}.

\begin{theorem}[Thorner--Zaman] \label{thm:chebotarev}
	There exist absolute and effectively computable constants $c_3,c_4,c_5$ such that the following holds.  Let $K/k$ be a normal extension of number fields with Galois group $G$.  Let $\beta_K$ denote the exceptional zero of $\zeta_K(s)$ and $\psi_{K/k}$ denote the exceptional character, if they exist. For any conjugacy class $C$ of $G$, define the class function $\mathbf{1}_C := \frac{1}{|G|}\sum_{\chi} \overline{\chi}(C) \chi$.
	 For any $x \geq ([K:\Q]^{[K:\Q]}  D_K)^{c_3}$, if we define 
		\begin{align*}
			\mathrm{Li}(x,C) & := \frac{|C|}{|G|}\big( \mathrm{Li}(x) - \psi_{K/k}(C) \mathrm{Li}(x^{\beta_K}) \big),   \\ 
			\delta(x,C) & := \frac{1}{\mathrm{Li}(x,C)} \Big( \sum_{\mathrm{N} \mathfrak{p} \leq x} \mathbf{1}_C(\mathfrak{p}) -\mathrm{Li}(x,C)  \Big),
		\end{align*}
		then  
		\[
		|\delta(x,C)| \leq  c_5 \exp\left(\frac{-c_4 \log x}{\log([K:\Q]^{[K:\Q]}D_K) + ([K:\Q] \log x)^{1/2}}\right).
		\]
\end{theorem}	
\begin{proof}
	This theorem is precisely \cite[Theorem 1.1]{TZ} upon noting that quantity $\theta_1(C)$ appearing there is $-\psi_{K/k}(C)$.  Indeed, Thorner and Zaman's notion of an exceptional zero, which they denote $\beta_1$ instead of $\beta_K$, is the same as Stark's (and hence ours).  They consider the zero $\beta_1$ as being associated with a real character of an abelian subgroup $H \leq G$, but it follows from Lemma~\ref{lem:stark-artin} that this character (if it exists) must be the restriction to $H$ of $\psi_{K/k}$.  This suffices to show that $\theta_1(C) = -\psi_{K/k}(C)$ and hence the theorem.
\end{proof}	

Due to its prevalence, it will be convenient to permanently define a shorthand for the error bound in \cref{thm:chebotarev}. For any number field $K$ and $t \geq ([K:\Q]^{[K:\Q]}  D_K)^{c_3}$,  define
\begin{equation}
	\label{eqn:pnt-error-bound}
	\Delta_K(t) := c_5  \exp\left(\frac{-c_4 \log t}{\log([K:\Q]^{[K:\Q]}D_K) + ([K:\Q] \log t)^{1/2}}\right)
\end{equation}
where the constants $c_3, c_4$ and $c_5$ are from \cref{thm:chebotarev}.  We proceed to asymptotically estimate subproducts of $L(1,\chi)$. 
\begin{lemma} \label[lemma]{lem:subproducts}
Let $K/k$ be a Galois extension with Galois group $G$. Let $\beta_K$ denote the exceptional zero of $\zeta_K(s)$ and let $\psi_{K/k}$ denote the exceptional character, if they exist.   For any character $\chi$ of $G$ and any $e([K:\Q]^{[K:\Q]} D_K)^{c_3} \leq y < x < \infty$,  
	\[
	\prod_{y < \mathrm{N}\mathfrak{p} \leq x} L_{\kp}(1,\chi) = \exp\Big( \int_y^x \frac{\langle \chi, \mathbf{1}_G \rangle}{t \log t} - \frac{\langle \chi, \psi_{K/k} \rangle}{t^{2-\beta_K} \log t} dt \Big) \Big\{ 1 + O \big( \Delta_K(y) \big) \Big\}^{\chi(1)}. 
	\]
\end{lemma}
\begin{proof} 
By \cref{lem:local-bound} and standard prime power arguments, we see that 
	\begin{equation} \label{eqn:subproduct-log}	 
	\begin{aligned}
	\sum_{y < \N\kp \leq x} \log L_{\kp}(1,\chi) 
		& = \sum_{y < \N\kp \leq x} \frac{\chi(\kp)}{\N\kp} + O\Big( \sum_{j=1}^{\infty} \sum_{y < p^j \leq x} \frac{[k:\Q] \chi(1)}{p^{2j}} \Big) \\
		&  = \sum_{y < \N\kp \leq x} \frac{\chi(\kp)}{\N\kp} + O\Big( \frac{[k:\Q] \chi(1)}{y} \Big).			
	\end{aligned}
	\end{equation}
	\normalsize
	The first error arises from the observations that $\N\kp = p^j$ for some rational prime $p$ and integer $j \geq 1$, and also that there are at most $[k:\Q]$ prime ideals $\kp$ above any given rational prime $p$.  The final error is absorbed by the claimed bound as $y \geq ([K:\Q]^{[K:\Q]} D_K)^{c_3}$ and $[K:\Q] \geq [k:\Q]$, so it suffices to analyze the remaining sum. 
		
	Define $S(t,\phi) = \sum_{\mathrm{N}\mathfrak{p} \leq t} \phi(\mathfrak{p})$ for any class function $\phi: G \to \mathbb{C}$. By \cref{thm:chebotarev} and \eqref{eqn:pnt-error-bound}, if $\mathbf{1}_C$ is the indicator function for a conjugacy class $C$ of $G$, it follows that
	\[
	S(t,\mathbf{1}_C) = \mathrm{Li}(t,C)  + \mathrm{Li}(t,C) \delta(t,C) \quad \text{ and }\quad|\delta(t,C)| \leq \Delta_K(t) \quad \text{ for } t \geq y.
	\]
	Expanding $\chi$ in terms of this basis for class functions, we see that 
	\[
	S(t,\chi) = \sum_{C} \chi(C) S(t,\mathbf{1}_C) = \sum_C \chi(C) \mathrm{Li}(t,C) + \sum_C \chi(C) \mathrm{Li}(t,C) \delta(t,C), 
	\]
	the summation over all conjugacy classes $C$ of $G$. By orthogonality of characters, we have 
	\[
	\sum_C \chi(C) \mathrm{Li}(t,C)= \langle \chi, \mathbf{1}_G \rangle \mathrm{Li}(t) - \langle \chi, \psi_{K/k} \rangle \mathrm{Li}(t^{\beta_K})
	\]
	and therefore, for $t \geq y$, 
	\[
	|S(t,\chi) - \langle \chi, \mathbf{1}_G \rangle \mathrm{Li}(t) + \langle \chi, \psi_{K/k} \rangle \mathrm{Li}(t^{\beta_K})| \leq  |\sum_C \chi(C) \mathrm{Li}(t,C) \delta(t,C)| \leq  2\chi(1) \mathrm{Li}(t) \Delta_K(t). 
	\]
	By partial summation and the bound $\mathrm{Li}(t) \ll t/\log t$, we have that 
	\[
	\sum_{y < \N\kp \leq x} \frac{\chi(\kp)}{\N\kp} 
	 = \int_y^x \frac{\langle \chi, \mathbf{1}_G \rangle}{t \log t} - \frac{\langle \chi, \psi_{K/k}\rangle}{t^{2-\beta_K} \log t} dt + O\Big(
	  \chi(1) \Big[ \int_y^x \frac{\Delta_K(t)}{t \log t} dt + \frac{\Delta_K(y)}{\log y} + \frac{\Delta_K(x)}{\log x} \Big] \Big) 
	 \]
	By a dyadic decomposition, the error term is at most  
	\[
	\ll \chi(1) \sum_{j=1}^{\infty} 	\int_{y^j}^{y^{j+1}} \frac{ \Delta_K(y^j) }{t \log t} dt \ll  \sum_{j=1}^{\infty}  \chi(1) \exp\Big( \frac{-c_4 j \log y}{\log(n^n D) + (jn \log y)^{1/2} } \Big), 
	\]
	which is dominated by the $j=1$ term, namely $\chi(1) \Delta_K(y)$. Upon combining with \eqref{eqn:subproduct-log} and exponentiating, this establishes the lemma. 
\end{proof}

This lemma yields a convenient form of Mertens' formula over number fields.  
	
\begin{lemma} \label[lemma]{lem:mertens} 
		Let $k$ be any number field. For any $x \geq y \geq \max\{ D_k, e\}$, 
		\[
		\sum_{y < \N\kp \leq x}\frac{1}{\N\kp} \leq \log\log x - \log\log y + O_{[k:\Q]}(1). 
		\]
\end{lemma}
\begin{proof} Set $n = [k:\Q]$ and $D = \max\{D_k,e\}$.  Since there are at most $n$ prime ideals above a rational prime $p$ and $\N\kp = p^j$ for some $j \geq 1$, it follows by Mertens' formula over $\Q$ that 
\[
\sum_{D < \N\kp < e(n^{n} D)^{c_3} } \frac{1}{\N\kp} \leq  \sum_{j=1}^{\infty} \sum_{D < p^j < e(n^n D)^{c_3} } \frac{n}{p^j} \ll n \log\Big( \frac{\log (n^n D)}{\log D} \Big) \ll_n 1. 
\]
The above estimate implies we may assume $y \geq e(n^n D)^{c_3}$ without loss of generality. Thus, from \cref{lem:subproducts} with $K=k$ and $\chi = \mathbf{1}_G$ and \eqref{eqn:subproduct-log}, we have that
\[
\sum_{y < \N\kp \leq x} \frac{1}{\N\kp} \leq \int_y^x \frac{1}{t \log t} dt + O( \Delta_k(y) + n/y) = \log\log x - \log\log y + O(1),  
\]
as required. 
\end{proof}
 
We conclude this section with the proof of \cref{thm:short-euler}. 

\begin{proof}[Proof of \cref{thm:short-euler}] The function 
	\[
	L(s) := \frac{L(s,\chi)}{\zeta(s)^{\langle \chi, \mathbf{1}_G\rangle} }
	\]
	is holomorphic and non-zero at $s=1$ by assumption. Moreover, since $(s-1)\zeta(s) \to 1$ as $s \to 1$, we have that 
	\[
	L(1) = \lim_{s \to 1} \big[ (s-1)^{\langle \chi,\mathbf{1}_G \rangle} L(s,\chi) \big] = \kappa(\chi)
	\]
	by \eqref{eqn:kappa-def}. Applying \cref{lem:subproducts} for $L(s,\chi)$ and for $\zeta(s)$ (which does not have any real zeros), we have for $3 ([K:\Q]^{[K:\Q]} D_K) \leq y < x < \infty$ that
		\[
		 \prod_{y < \N\kp \leq x} L_{\kp}(1,\chi)  \times  \prod_{y < p \leq x} \Big(1  - \frac{1}{p} \Big)^{\langle \chi, \mathbf{1}_G \rangle}
		 = \exp\Big( - \int_y^x \frac{\langle \chi, \psi_{K/k}\rangle}{t^{2-\beta_K} \log t} dt \Big) \times \Big \{ 1 + O\big(  \Delta_K(y) \big) \Big\}^{\chi(1)} 
		\]
		since $\Delta_{\Q}(y) \leq \Delta_K(y)$ by   \eqref{eqn:pnt-error-bound}. Recall  if $\beta_K$ does not exist then we treat $\psi_{K/k} \equiv 0$ by convention. Taking $y = T$ and $x \to \infty$ above, we conclude that 
		\[
		L(1) 
		=  \eta(\chi,T) \prod_{\N\kp \leq T} L_{\kp}(1,\chi)  \times  \prod_{p \leq T} \Big(1  - \frac{1}{p} \Big)^{\langle \chi, \mathbf{1}_G \rangle}  \Big\{ 1 +  O\big( \Delta_K(T)\big) \Big\}^{\chi(1)},			
		\]
		where $\eta(\chi,T)$ is defined in \cref{thm:short-euler}. 	After applying the prime number theorem over $\Q$ for the product over primes $p \leq T$ in the form
		\[
		\prod_{p \leq T} \Big(1  - \frac{1}{p} \Big) = \frac{e^{-\gamma}}{\log T}\Big( 1 + O(e^{-\sqrt{\log T}}) \Big),
		\]
		we complete the proof. Here $\gamma$ is the Euler-Mascheroni constant over $\Q$. 
\end{proof}

%
%
%
%	SECTION 7: THREE KEY PROPOSITIONS
%
%
%

\section{Three key propositions} \label{sec:key-props}

To derive upper and lower bounds on $\kappa(\chi)$, we establish three key propositions.  Our first key proposition approximates $\kappa(\chi)$ by short Euler products with \textit{variable} lengths $T(\psi)$ depending on each irreducible component $\psi$ of $\chi$.  

\begin{proposition} \label[proposition]{prop:residue}
		Let $K/k$ be a Galois extension of number fields with Galois group $G$. Let $( T(\psi) )_{\psi}$ be a tuple of real numbers such that $T(\psi) \geq \max\{ q(\psi), e\}$ for every irreducible character $\psi$ of $G$. For any character $\chi$ of $G$,
	\[
	\kappa(\chi) \asymp_{[k:\Q],|G|,\chi(1)} \frac{1}{(\log T(\mathbf{1}_G))^{\langle \chi, \mathbf{1}_G\rangle}} \prod_{\psi} \Big( \eta(\psi, T(\psi) ) \prod_{\N\kp \leq T(\psi)} L_{\kp}(1,\psi) \Big)^{\langle \chi, \psi \rangle}, 
	\]
	where, for any  irreducible character $\psi$ of $G$ and real number $\tau \geq 3$,
	\begin{equation} \label{eqn:eta-T-variable-def}
	\eta(\psi, \tau) = 
		\begin{cases}
		 		\exp\Big( -\displaystyle \int_{\tau}^{\infty} \frac{1}{t^{2-\beta_{\psi}} \log t} dt \Big)   & \text{if $\psi \in \Psi_{K/k}(G)$,} \\[10pt]
		 		1 & \text{otherwise.}
		 \end{cases} 		
	\end{equation}
	\normalsize
\end{proposition} 

\begin{proof} It suffices to show for any irreducible character $\psi$ of $G$ that
\begin{equation} \label{eqn:prop-residue-irreducible}
	\kappa(\psi)  \asymp_{[k:\Q],|G|,\psi(1)}\frac{\eta(\psi, T(\psi))}{(\log T(\psi))^{\langle \psi, \mathbf{1}_G \rangle}} \prod_{\N\kp \leq T(\psi)} L_{\kp}(1,\psi), 
\end{equation}
because $\kappa(\chi) = \prod_{\psi} \kappa(\psi)^{\langle \chi, \psi \rangle}$. Now, fix an irreducible character $\psi$. Define the subfield $K(\psi) = K^{\ker \psi}$ of $K$, so $\psi$ is a faithful character of the Galois extension $K(\psi)/k$ with Galois group $G/\ker \psi$. Set $n(\psi) = [K(\psi):\Q]$ and $D(\psi) = D_{K(\psi)}$ temporarily. 

Apply \cref{thm:short-euler} to $\psi$ as a character of the extension $K(\psi)/k$ with truncation parameter $\widetilde{T}(\psi) = \max\{T(\psi), 3(n(\psi)^{n(\psi)} D(\psi))^{c_3} \}$. This yields 
\[
\kappa(\psi) \asymp_{n(\psi),\psi(1)} \frac{\eta(\psi, \widetilde{T}(\psi))}{(\log \widetilde{T}(\psi))^{\langle \psi, \mathbf{1}_G \rangle}} \prod_{\N\kp \leq \widetilde{T}(\psi)} L_{\kp}(1,\psi). 
\]
It remains to replace every instance of $\widetilde{T}(\psi)$ with $T(\psi)$ in the above estimate because \eqref{eqn:prop-residue-irreducible} would then follow from \cref{lem:local-bound}. 

If $\widetilde{T}(\psi) = T(\psi)$ then we are done.  Otherwise, $q(\psi) \leq T(\psi) \leq \widetilde{T}(\psi) = e(n(\psi)^{n(\psi)} D(\psi) )^{c_3}$ by assumption. Since $\psi$ is a faithful character of $K(\psi)/k$, we have by \cref{lem:conductor-bound} that $\log q(\psi) \asymp_{n(\psi),\psi(1)} \log \widetilde{T}(\psi)$. This implies
\[
\log T(\psi) \asymp_{n(\psi),\psi(1)} \log \widetilde{T}(\psi) 
\]
and hence, by trivially bounding the product via \cref{lem:local-bound}  and using Mertens' formula over $\Q$, we have 
\[
 \prod_{T(\psi) < \N\kp \leq \widetilde{T}(\psi)} L_{\kp}(1,\psi)  \asymp_{n(\psi),\psi(1)} 1. 
\]
Finally, we must show $\eta(\psi,\widetilde{T}(\psi)) \asymp_{n(\psi),\psi(1)} \eta(\psi, T(\psi))$ assuming $q(\psi) \leq T(\psi) \leq \widetilde{T}(\psi)$. We need only consider when $\psi^2 = \mathbf{1}_G$ and $L(s,\psi)$ has a real zero $\beta = \beta_{\psi} > 1 - \frac{1}{4 \log q(\psi)}$. Observe by monotonicity and non-negativity that
\[
0 \leq \int_{T(\psi)}^{\widetilde{T}(\psi)} \frac{1}{t^{2-\beta} \log t} dt \leq  \int_{T(\psi)}^{\widetilde{T}(\psi)} \frac{1}{t\log t} dt =  \log\Big( \frac{\log \widetilde{T}(\psi) }{\log T(\psi) } \Big) \ll_{n(\psi),\psi(1)} 1. 
\]
This proves $\eta(\psi,\widetilde{T}(\psi)) \asymp_{n(\psi),\psi(1)} \eta(\psi, T(\psi))$ in all cases, as required. Collecting all of our observations completes the proof. 
\end{proof}

Our second key proposition estimates short Euler products with variable lengths.  
 
\begin{proposition} \label[proposition]{prop:small-primes}
	Let $K/k$ be a Galois extension of number fields with Galois group $G$.  Index all  irreducible characters $\psi_1,\dots,\psi_N$ of $G$ and assume $\psi_N = \mathbf{1}_G$. Let $T_1,\dots,T_N \geq e$ be real numbers such that $T_1 \geq \cdots \geq T_{N}$ and $T_n \geq q(\psi_n)$ for $1 \leq n \leq N$. Let $\chi$ be any character of $G$ and let $\widetilde{\chi}$ be the induction of $\chi$ to the Galois closure of $K$ over $\mathbb{Q}$. Define $\chi_n := \sum_{i=1}^n \langle \chi, \psi_i \rangle \psi_i$ for $1 \leq n \leq N, \chi_0 \equiv 0$, and
	\[
	P := \prod_{n=1}^N \prod_{\N\kp \leq T_n} L_{\kp}(1,\psi_n)^{\langle \chi, \psi_n \rangle}.
	\]

	Then 
	\[
	  |P| \ll_{[k:\Q],|G|,\chi(1)} (\log T_N)^{\widetilde{\chi}(1) - \chi(1)} \prod_{n=1}^{N} ( \log T_n )^{\chi_n(1)-\chi_{n-1}(1)}
	\]
	and 
	\[
	|P| \gg_{[k:\Q],|G|,\chi(1)} (\log T_N)^{\mu(\widetilde{\chi}) - \mu(\chi)} \prod_{n=1}^N ( \log T_n )^{\mu(\chi_n)-\mu(\chi_{n-1})}.
	\]
\end{proposition}

\begin{remark}
Notice the upper bound does not depend on the ordering of $\psi_1,\dots,\psi_N$, because the map $\chi \mapsto \chi(1)$ is linear and hence $\chi_{n}(1) - \chi_{n-1}(1) = \langle \chi, \psi_n \rangle \psi_n(1)$. On the other hand, the lower bound may depend on the ordering because the map $\chi \mapsto \mu(\chi)$ is sublinear and hence  $\mu(\chi_n) - \mu(\chi_{n-1}) \geq \langle \chi, \psi_n \rangle \mu(\psi_n)$. 	
\end{remark}
\begin{remark}
	Note $\widetilde{\chi}(1) = [k:\Q] \chi(1)$ and, if $\mu(\chi) < 0$, then $\mu(\widetilde{\chi}) \geq [k:\Q] \mu(\chi)$. 
\end{remark}
\begin{remark}
The $n=N$ contribution to the upper bound is
\[
(\log T_N)^{\widetilde{\chi}(1)- \chi(1) + \chi_N(1) -\chi_{N-1}(1)} = (\log T_N)^{\widetilde{\chi}(1) - \chi_{N-1}(1)}
\]
since $\chi = \chi_N$. The same follows for the lower bound. We have included the $n=N$ term in the product so that the telescoping cancellation will be more apparent when proving our main theorem. 
\end{remark}

\begin{proof}   Define $T_{N+1} := 1$ and, for $1 \leq n \leq N$,
	\[
	P_n := \prod_{T_{n+1} < \N\kp \leq T_n} L_{\kp}(1,\chi_n), \quad \text{ so } \quad P_n = \prod_{i=1}^n \prod_{T_{n+1} < \N\kp \leq T_n} L_{\kp}(1,\psi_i)^{\langle \chi, \psi_i\rangle}
	\] 
	by linearity of Artin $L$-functions and the definition of $\chi_n$. 	It follows that 
	\[
	P = P_1 \cdots P_N . 
	\]
	For $1 \leq n \leq N-1$, we have $\langle \chi_n, \psi_N \rangle= \langle \chi_n, \mathbf{1}_G \rangle =0$ by construction, so $\mu(\chi_n) < 0 < \chi_n(1)$. Note $\mu(\chi_n) < 0$ since $\sum_{g \in G} \chi_n(g) = \langle \chi_n,\mathbf{1}_G \rangle = 0$ and $\chi_n(1) > 0$.   Therefore, \cref{lem:local-bound,lem:mertens} imply that 
	\[
|P_n| \leq \prod_{T_{n+1} < \N\kp \leq T_n} \exp\Big(\frac{\chi_n(1)}{\N\kp} + \frac{2\chi_n(1)}{\N\kp^2} \Big)  \ll_{[k:\Q],\chi(1)} \Big( \frac{\log T_n}{\log T_{n+1}}\Big) ^{\chi_n(1)}
	\]
	and, as $\mu(\chi_n) < 0$, 
	\[
	|P_n| \geq  \prod_{T_{n+1} < \N\kp \leq T_n} \exp\Big(\frac{\mu(\chi_n)}{\N\kp} - \frac{2\chi_n(1)}{\N\kp^2} \Big) \gg_{[k:\Q],\chi(1)} \Big( \frac{\log T_n}{\log T_{n+1}}\Big) ^{\mu(\chi_n)}. 
	\]
	For $n=N$, we have $\psi_N = \mathbf{1}_G$ and $\chi_N = \chi$. We shall bound the product $P_N$ trivially using rational primes.  Since $\widetilde{\chi}(p) = \sum_{\N\kp = p} \chi(\kp)$, \cref{lem:local-bound} implies that 
	\[
	\log |P_{N}| = \sum_{\N\kp \leq T_{N}}   \frac{\chi(\kp)}{\N\kp} + O\Big(  \sum_{\N\kp \leq T_{N}} \frac{\chi(1)}{\N\kp^2} \Big) = \sum_{p \leq T_N} \frac{\widetilde{\chi}(p)}{p} + O( \chi(1) [k:\Q] ). 
	\]
	As $\mu(\widetilde{\chi}) \leq \widetilde{\chi}(p) \leq \widetilde{\chi}(1)$ from \cref{lem:local-bound} again, it follows by Mertens' formula over $\Q$ that 
\[
(\log T_{N})^{\mu(\widetilde{\chi})} \ll_{[k:\Q],\chi(1)} |P_N|   \ll_{[k:\Q],\chi(1)} (\log T_{N})^{\widetilde{\chi}(1)}. 
\]	
Collecting our observations yields the result. 
\end{proof}
\begin{remark} For $1 \leq n  < N$, the condition that $\mu(\chi_n) < 0$ is critically used in the proof to apply an asymptotically sharp \textit{upper} bound on $\sum_{T_{n+1} < \N\kp \leq T_n} 1/\N\kp$  via \cref{lem:mertens}. A sharp lower bound for this quantity is unavailable precisely due to the potential presence of a Landau--Siegel zero for $\zeta_k(s)$. 	For $n=N$, we rewrite the product over rational primes and appeal to Mertens' asymptotic formula over $\Q$ which is valid for all values of $\mu(\widetilde{\chi}) \in \mathbb{R}$. 
\end{remark}

Our third and final key proposition provides estimates for $\eta(\psi,T(\psi))$ and $\eta(\psi)$  for the exceptional characters $\psi \in \Psi_{K/k}(G)$, which are defined in \eqref{eqn:eta-T-variable-def}, \eqref{eqn:eta-def}, and \eqref{eqn:exceptional-quadratics} respectively.

\begin{proposition} \label[proposition]{prop:eta}
		Let $K/k$ be a Galois extension of number fields with Galois group $G$, and $\Psi_{K/k}(G)$ defined by \eqref{eqn:exceptional-quadratics}. Let $ ( T(\psi) )_{\psi}$ be a tuple of real numbers indexed by the irreducible characters $\psi$ of $G$  such that $T(\psi) \geq \max\{ q(\psi), e\}$. Let $\eta(\psi)$ and $\eta(\psi,T(\psi))$ be defined by \eqref{eqn:eta-def} and \eqref{eqn:eta-T-variable-def} respectively. 
		 All of the following hold: 
		\begin{enumerate}[label=(\roman*)]
			\item  For every $\psi \in \Psi_{K/k}(G)$, we have  $0 <\eta(\psi), \eta(\psi,T(\psi))  \leq 1$ and $\eta(\psi) \leq 2   \eta(\psi,T(\psi)).$ 
			\item For every $\psi \in \Psi_{K/k}(G)$, if $T(\psi) \leq q(\psi)^A$ for some $A \geq 1$ then  $\eta(\psi,T(\psi)) \leq e^{A}  \eta(\psi)$. 
			\item For every $\psi \in \Psi_{K/k}(G)$, we have 
			\[
				\eta(\psi) \gg_{[k:\Q]} \frac{\log q(\psi) }{(D_k q(\psi))^{1/2[k:\Q]} }. 
			\]
			\item For any character $\chi$ of $G$,  we have 
				\[
				\prod_{\psi} \eta(\psi)^{\langle \chi, \psi\rangle} \gg_{[k:\Q], \chi(1)}  \min\Big\{ \Big( \frac{\log q(\psi) }{(D_k q(\psi))^{1/2[k:\Q]} }   \Big)^{\langle \chi, \psi \rangle} : \psi \in \Psi_{K/k}(G)  \Big\}. 
				\]
		\end{enumerate}
\end{proposition}
\begin{remark}
In light of (iii), it might be surprising that (iv) is possible, since (iv) essentially concentrates the lower bound on a single character's worst-case contribution. This feature is achieved by carefully exploiting zero repulsion effects between real zeros from \cref{lem:deuring-heilbronn}. 
\end{remark}
\begin{proof} For (i), the claim $0 < \eta(\psi), \eta(\psi,T(\psi)) \leq 1$ is immediate  since $0 < (1-\beta_{\psi}) \log q(\psi) < 1/4$ by definition of $\Psi_{K/k}(G)$.  For the other claim, denote $T = T(\psi)$ and $\beta = \beta_{\psi}$ for simplicity. By a dyadic decomposition, we have that
	\begin{equation} \label{eqn:eta-T-dyadic}
	\eta(\psi,T) = 
		\exp\Big( -   \int_{T}^{\infty} \frac{1}{t^{2-\beta} \log t} dt \Big)  = \exp\Big( - \sum_{j=1}^{\infty} \int_{T^j}^{T^{j+1}} \frac{1}{t^{2-\beta} \log t} dt \Big).
	\end{equation}
	By monotonicity, the infinite sum of integrals is at most
	\[
	 \sum_{j=1}^{\infty}  T^{-j(1-\beta)} 	\int_{T^j}^{T^{j+1}}\frac{1}{t \log t} dt	 \leq \sum_{j=1}^{\infty}  \frac{1}{j}T^{-j(1-\beta)} = - \log\big(1- T^{-(1-\beta)} \big)
	\]
	since $\log(1+\frac{1}{j}) \leq \frac{1}{j}$ and $\sum_{j=1}^{\infty} u^j/j = -\log(1-u)$ for $0 < u < 1$. Therefore, as $T \geq q(\psi)$, 
	\[
	\eta(\psi,T)  
	\geq  1- q(\psi)^{-(1-\beta)} \geq \frac{1}{2}  (1-\beta) \log q(\psi)   = \frac{1}{2} \eta(\psi)
	\]
	since $1-e^{-u} \geq u/2$ for $0 < u < 1/4$. This establishes (i). 	
	
	For (ii), the argument is similar to (i). 	Set $\beta = \beta_{\psi}, T = T(\psi),$ and $T_A = \exp(A/(1-\beta))$ for simplicity, so $T_A \geq q(\psi)^{4A} \geq T^4$ by assumption. As $e^{-u} \geq 1-u$ for $u > 0$, it follows that 
	\[
	\int_{T}^{T_A} \frac{1}{t^{2-\beta} \log t} dt \geq \int_{T}^{T_A} \frac{1}{t \log t} dt - (1-\beta) \int_{T}^{T_A} \frac{
	\log t}{t \log t}dt  \geq \log\Big(\frac{\log T_A}{\log T} \Big) - A. 
	\]
	From \eqref{eqn:eta-T-dyadic} and non-negativity, it follows that
	\[
	\eta(\psi,T) \leq \exp\Big( - \int_{T}^{T_A} \frac{1}{t^{2-\beta} \log t} dt \Big) \leq   \frac{e^A \log T}{\log T_A}  
	\leq  e^A (1-\beta) \log q(\psi) = e^A \eta(\psi),
	\]
	as required. This establishes (ii). 

	For (iii), this follows immediately from \cref{lem:stark-effective} and the definition of $\Psi_{K/k}(G)$. 
		
	For (iv), let $\psi_1,\dots,\psi_N$ be the complete list of characters belonging to $\Psi_{K/k}(G)$ such that $\langle \chi, \psi_i \rangle \geq 1$ for every $1 \leq i \leq N$. For $1 \leq i \leq N$, denote the analytic conductor by $q_i = q(\psi_i) \geq 2$ and its real zero by $\beta_i = \beta_{\psi_i}$, i.e. $L(\beta_i,\psi_i) = 0$ and $\beta_i > 1-\frac{1}{4 \log q_i}$. Without loss of generality, assume that $\beta_1 \geq \cdots \geq \beta_N$. By definition of $\eta(\psi)$, it follows that 
	\begin{equation} \label{eqn:eta-factored}
	\prod_{\psi} \eta(\psi)^{\langle \chi, \psi\rangle} = 	\prod_{i=1}^N \eta(\psi_i)^{\langle \chi, \psi_i\rangle} = \prod_{i=1}^N  \big( (1-\beta_i) \log q_i \big)^{\langle \chi, \psi_i\rangle}. 
	\end{equation} 
	Fix $2 \leq i \leq N$. We shall give a lower bound for $(1-\beta_i) \log q_i$ in two cases. 
	\begin{itemize}
		\item Assume $q_1 < q_{i}$. 	Since $\beta_1 \geq \beta_{i}$,   \cref{lem:deuring-heilbronn}(i)   implies that 
	\[
	(1-\beta_{i}) \log q_{i} \geq  \frac{1}{2} (1-\beta_{i}) \log(q_1 q_{i}) > \frac{1}{24}.  
	 \]
	 	\item Assume $q_1 \geq q_{i}$. Since $\beta_1 \geq \beta_{i}$, \cref{lem:deuring-heilbronn}(ii) implies that 
	\begin{align*}
	(1-\beta_i) \log q_i 
	& \geq \frac{\log q_i}{\log(q_1 q_i)} \max\Big\{ \frac{1}{12} ,  c_1^{-1} \log\Big(\frac{c_2}{(1-\beta_1) \log( q_1 q_i) } \Big) \Big\} \\ 
	& \gg \frac{1}{\log q_1} \log\Big(\frac{1}{(1-\beta_1)\log q_1} \Big)
	\end{align*}
	\end{itemize}
	Overall, these cases and \eqref{eqn:eta-factored} imply that 
	\[
	\prod_{\psi} \eta(\psi)^{\langle \chi, \psi\rangle} \gg_{\chi(1)} \big( (1-\beta_1) \log q_1 \big)^{\langle \chi, \psi_1\rangle}  \min\Big\{ 1, \frac{1}{\log q_1}  \log\Big(\frac{1}{(1-\beta_1) \log q_1 } \Big) \Big\}^{\chi(1)}. 
	\]
	Since $\langle \chi, \psi_1 \rangle \geq 1$, the righthand expression is minimized when $\beta_1$ is maximized. From \cref{lem:stark-effective}, we conclude that
	\[
	\prod_{\psi} \eta(\psi)^{\langle \chi, \psi\rangle} \gg_{\chi(1), [k:\Q]} \Big(\frac{\log q_1}{(D_k q_1)^{1/2[k:\Q]} }\Big)^{\langle \chi, \psi_1 \rangle},
	\]
	which establishes (iv).  
\end{proof}

We conclude this section by establishing \cref{prop:residue-optimal}. 
\begin{proof}[Proof of \cref{prop:residue-optimal}]
The proposition follows from \cref{prop:residue} with $T(\psi) = q(\psi)$ for every nontrivial $\psi$ and $T(\mathbf{1}_G) = e D_k$, and the observation $\eta(\psi) \asymp \eta(\psi,T(\psi))$ from \cref{prop:eta}(i) and (ii). 
\end{proof}

%
%
%
%	SECTION 8: MAIN PROOFS
%
%
%

\section{Proofs of Theorems 1.1 and 1.2 and Corollaries 1.3 and 1.4} \label{sec:proofs-main}

For all proofs, let $K/k$ be a Galois extension of number fields with Galois group $G$. Let $\Psi_{K/k}(G) \subseteq \mathrm{Irr}(G)$ be the set of potentially exceptional characters associated to $K/k$ defined by \eqref{eqn:exceptional-quadratics}. Let $\chi$ be any character of $G$ and let $\widetilde{\chi}$ be its induction to the Galois closure $\widetilde{K}$ of $K$ over $\Q$. If $\widetilde{\chi}$ is an integer multiple of the trivial character over $\Q$, then $L(s,\chi) = L(s,\widetilde{\chi}) = \zeta(s)^m$ for some integer $m \geq 1$, so $\kappa(\chi) = 1$ and there is nothing to prove. We may therefore assume for all proofs that $\widetilde{\chi}$ includes some nontrivial component and hence $q(\chi) = q(\widetilde{\chi}) \geq 3$. This lower bound on conductors can be deduced, for example, from Minkowski's classical lower bound on discriminants or from work of Odlyzko \cite{Odlyzko}.

\begin{proof}[Proof of \cref{thm:disc}] Let $A \geq 1$ be sufficiently large, depending at most on $[k:\Q], |G|,$ and $\chi(1)$. 	From \cref{thm:short-euler} with $T = 3([K:\Q]^{[K:\Q]} D_K)^{A}$, we have that  
	\begin{equation} \label{eqn:proof-disc-step1}
	\kappa(\chi) \asymp_{[k:\Q],|G|,\chi(1)} \frac{\widetilde{\eta}(\chi,T)}{(\log T)^{\langle \chi, \mathbf{1}_G \rangle}} \prod_{\N\kp \leq T} L_{\kp}(1,\chi). 
	\end{equation}
	Denoting $N =|\mathrm{Irr}(G)|$, we apply \cref{prop:small-primes} with $T_1 = \cdots = T_N = T$. The telescoping nature of the exponents yields
	\begin{equation} \label{eqn:proof-disc-step2}
	(\log T)^{\mu(\widetilde{\chi})} \ll_{[k:\Q],|G|,\chi(1)}   \prod_{\N\kp \leq T} L_{\kp}(1,\chi) \ll_{[k:\Q],|G|,\chi(1)} (\log T)^{\widetilde{\chi}(1)} 
	\end{equation}
	 If $\psi_{K/k}$ does exist, then $\widetilde{\eta}(\chi,T) = \eta(\psi_{K/k},T)^{\langle \chi, \psi_{K/k}\rangle}$ by comparing definitions with \eqref{eqn:eta-T-variable-def}. By \cref{prop:eta}(i) and (iii), it follows that 
	\[
	\Big( \frac{\log q(\psi_{K/k})}{(D_k q(\psi_{K/k}))^{1/2[k:\Q]} )} \Big)^{\langle \chi, \psi_{K/k} \rangle} \ll_{[k:\Q], |G|, \chi(1)} (\tfrac{1}{2}\eta(\psi_{K/k}))^{\langle \chi, \psi_{K/k}\rangle} \leq \widetilde{\eta}(\chi,T) \leq 1. 
	\]
	If $F$ is the quadratic or trivial extension of $k$ defined by $\psi_{K/k}$, then $D_k q(\psi_{K/k}) = D_F \leq D_K^{2/|G|}$ when $F$ is quadratic and $D_k q(\psi_{K/k}) = D_F^2 \leq D_K^{2/|G|}$ when $F$ is trivial. Either way, as the function $x \mapsto \frac{\log x}{x^{1/m}}$ for a positive constant $m>0$ is decreasing for $x \geq e^m$, we have whenever $D_k \geq e^{[k:\Q]}$ that 
	\[
\frac{\log q(\psi_{K/k})}{(D_k q(\psi_{K/k}))^{1/2[k:\Q]} )} \gg_{[k:\Q],|G|} \frac{\log D_K}{D_K^{1/[K:\Q]} }.
\]
For fields $k$ with $D_k < e^{[k:\Q]}$, the inequality above holds trivially. Therefore, we have that 
\begin{equation} \label{eqn:proof-disc-step3}
\Big( \frac{\log D_K}{ D_K^{1/[K:\Q]} } \Big)^{\nu(\chi)} \ll_{[k:\Q],|G|, \chi(1)} \widetilde{\eta}(\chi,T) \leq 1, 
\end{equation}
	where $\nu(\chi)$ is defined in \cref{thm:disc}. 
	
If $\psi_{K/k}$ does not exist, then $\widetilde{\eta}(\chi,T) = 1$ so the above bound is still valid. Since  $\log T \asymp_{[k:\Q],|G|,\chi(1)} \log D_K$, estimates \eqref{eqn:proof-disc-step1}, \eqref{eqn:proof-disc-step2}, and \eqref{eqn:proof-disc-step3} complete the proof.  
\end{proof}

\begin{proof}[Proof of \cref{thm:cond}]
	First, we invoke \cref{prop:residue} with $T(\psi) = q(\chi) \geq 3$ for every nontrivial irreducible character $\psi$ of $G$ and $T(\mathbf{1}_G) = D := \max\{D_k,3\}$. This yields  
	\begin{equation}
	\label{eqn:proof-main-residue}
	\kappa(\chi) \asymp_{[k:\Q],|G|,\chi(1)} \frac{E \cdot P}{(\log eD_k)^{\langle \chi, \mathbf{1}_G \rangle} }, 
	\end{equation}
	where
	\[
	E := \prod_{\psi} \eta(\psi,T(\psi))^{\langle \chi, \psi \rangle}  \quad \text{ and } \quad P := \prod_{\N\kp \leq D} |L_{\kp}(1,\mathbf{1}_G)|^{\langle \chi, \mathbf{1}_G\rangle} \times \prod_{\psi \neq \mathbf{1}_G} \prod_{\N\kp \leq q(\chi) } |L_{\kp}(1,\psi)|^{\langle \chi, \psi\rangle}.
	\]
	
Next, denoting $N$ to be the number of irreducible characters of $G$, we apply \cref{prop:small-primes} to the product $P$ with $T_1 = \cdots = T_{N-1} = q(\chi)$ and $T_N = D$.  Here we have used that $q(\chi) \geq D$ and hence $T_{N-1} \geq T_N$ with $\psi_N = \mathbf{1}_G$. By the telescoping nature of the resulting bounds and the observation that $\chi_N = \chi$, we deduce that
\begin{equation}
\label{eqn:proof-main-P-upper}
P \ll_{[k:\Q],|G|,\chi(1)} (\log eD_k)^{\widetilde{\chi}(1)-\chi_{N-1}(1)} (\log q(\chi) )^{\chi_{N-1}(1)}
\end{equation}
and 
\begin{equation}
\label{eqn:proof-main-P-lower}
P \gg_{[k:\Q],|G|,\chi(1)} (\log eD_k)^{\mu(\widetilde{\chi})-\mu(\chi_{N-1})} (\log q(\chi) )^{\mu(\chi_{N-1})}, 
\end{equation}
where $\chi_{N-1} = \chi - \langle \chi, \mathbf{1}_G \rangle \mathbf{1}_G$. As $\mathbf{1}_G$ is constant, notice that
\begin{equation} \label{eqn:proof-Nminus1}
\chi_{N-1}(1) = \chi(1) - \langle \chi,\mathbf{1}_G \rangle \quad \text{ and } \quad \mu(\chi_{N-1}) = \mu(\chi) - \langle \chi,\mathbf{1}_G \rangle. 
\end{equation}
Finally, we estimate $E$. For the upper bound, we have $E \leq 1$ by \cref{prop:eta}(i).   Combined with \eqref{eqn:proof-main-residue}, \eqref{eqn:proof-main-P-upper}, and \eqref{eqn:proof-Nminus1}, this implies the desired upper bound. For the lower bound, we have $E \gg_{[k:\Q],\chi(1)} \epsilon(\chi)$ by \cref{prop:eta}(iii), where $\epsilon(\chi)$ is defined by \eqref{eqn:epsilon-def}. Combined with \eqref{eqn:proof-main-residue}, \eqref{eqn:proof-main-P-lower}, and \eqref{eqn:proof-Nminus1}, this yields the desired lower bound. This completes the proof. 
\end{proof} 

\begin{proof}[Proof of \cref{cor:nonexceptional}]
	This corollary follows immediately from \cref{thm:cond} since our assumption implies $\langle \chi, \psi \rangle = 0$ for every $\psi \in \mathrm{Irr}(G)$ with $\psi^2 = \mathbf{1}_G$. 
\end{proof}

\begin{proof}[Proof of \cref{cor:irreducibles}]
	Applying \cref{thm:cond} to each term in the identity $\kappa(\chi) = \prod_{\psi} \kappa(\psi)^{\langle \chi, \psi \rangle}$ and noting $\mathbf{1}_G(1) = \mu(\mathbf{1}_G) = 1$, we find that 
		\begin{equation} \label{eqn:irred-upper} \small
		|\kappa(\chi)| \ll_{[k:\Q],|G|,\chi(1)}  (\log eD_k)^{(\widetilde{\mathbf{1}_G}(1) -1) \langle\chi,\mathbf{1}_G\rangle} \prod_{\psi \neq \mathbf{1}_G} \big[ (\log eD_k)^{\widetilde{\psi}(1)-\psi(1)} (\log q(\psi))^{\psi(1)} \big]^{\langle \chi, \psi\rangle}
		\end{equation}
		and
		\begin{equation}\label{eqn:irred-lower} \small 
		|\kappa(\chi)| \gg_{[k:\Q],|G|,\chi(1)} \epsilon(\chi)  (\log eD_k)^{(\mu(\widetilde{\mathbf{1}_G}) - 1)\langle \chi, \mathbf{1}_G \rangle} \prod_{\psi \neq \mathbf{1}_G} \big[ (\log eD_k)^{\mu(\widetilde{\psi})-\mu(\psi)}   (\log q(\psi))^{\mu(\psi)} \big]^{\langle \chi, \psi \rangle}.
		\end{equation}
		The corollary now follows from the observations that $\mu(\psi) \geq -\psi(1)$ and $\widetilde{\psi}(1) = [k:\Q] \psi(1)$ for all $\psi$, $\mu(\widetilde{\mathbf{1}_G}) \geq 0$, $\mu(\widetilde{\psi}) \geq [k:\Q] \mu(\psi)$ for all nontrivial $\psi$, and $\chi = \sum_{\psi} \langle \chi, \psi \rangle \psi$.
\end{proof}

\begin{remark} 
	The ``decomposed'' upper bound in \eqref{eqn:irred-upper} is uniformly better than the ``undecomposed'' upper bound in \cref{thm:cond} since $\chi \mapsto \chi(1)$ is a linear map and, for positive real numbers $x_i$ and positive integers $a_i$, we have  $ x_1^{a_1} \cdots x_N^{a_N} \ll_{a_1,\dots,a_N} (a_1x_1+\cdots + a_Nx_N)^{a_1+\cdots + a_N}$. In particular,   the upper bound on $\kappa(\chi)$ in \cref{cor:irreducibles} is uniformly better than its  upper bound in \cref{thm:cond}. 
\end{remark}

%
%
%
%	SECTION 8: GRH PROOF
%
%
%

\section{Proof of Proposition 2.1} \label{sec:grh}
Assume GRH for $\zeta_K(s)$. For any conjugacy class $C$ of $G = \mathrm{Gal}(K/k)$,  a conditional version of the Chebotarev density theorem due to Lagarias and Odlyzko \cite{LO} implies that 
\[
\sum_{\N\kp \leq x} \mathbf{1}_C(\kp) = \frac{|C|}{|G|} \mathrm{Li}(x) + O_{[k:\Q],|G|}( x^{1/2} \log(D_K x) )
\]
for $x \geq (\log D_K)^2$. Therefore, for any irreducible character $\chi$ of $G$, it follows by orthogonality of characters that 
\begin{equation} \label{eqn:pnt-grh}
\sum_{\N\kp \leq x} \chi(\kp) = \langle \chi, \mathbf{1}_G \rangle \mathrm{Li}(x) + O_{[k:\Q],|G|,\chi(1)}( x^{1/2} \log( D_K x) )
\end{equation}
for $x \geq (\log D_K)^2$.  By applying partial summation to the trivial character, we may replace \cref{lem:mertens} with the conditional estimate
\begin{equation} \label{eqn:mertens-grh}
\sum_{y < \N\kp \leq x} \frac{1}{\N\kp} = \log\log x - \log\log y + O_{[k:\Q]}(1) 
\end{equation}
for $x \geq y \geq \log D_k$. Note the primes between $\log D_k$ and $(\log D_k)^2$ are discarded trivially.

Now, following the same arguments as \cref{lem:subproducts} with \eqref{eqn:pnt-grh}, we deduce that
\[
\prod_{y < \N\kp \leq x} L_{\kp}(1,\chi) = \exp\Big( \int_y^x \frac{\langle \chi, \mathbf{1}_G \rangle}{t \log t} dt \Big) \Big\{ 1 + O\big( y^{-1/2} \log(D_K y) \big) \Big\}^{\chi(1)}
\]
for $x > y > (\log D_K)^2$. Continuing with the arguments in the proof of \cref{thm:short-euler} (appearing at the end of \S\ref{sec:short-euler}), we similarly deduce that
\[
\kappa(\chi) = \frac{1}{(e^{\gamma} \log T)^{\langle \chi, \mathbf{1}_G \rangle}} \Big( \prod_{\N\kp \leq T} L_{\kp}(1,\chi) \Big) \Big\{ 1 + O( T^{-1/2} \log (D_KT) ) \Big\}^{\chi(1)}
\]
for $T \geq (\log D_K)^2$. By applying this estimate to faithful characters and their corresponding subextensions (as we did in \cref{prop:residue}), we find that
\[
\kappa(\chi) \asymp_{[k:\Q],|G|,\chi(1)} \frac{1}{(\log T(\mathbf{1}_G))^{\langle \chi, \mathbf{1}_G\rangle}} \prod_{\psi}   \prod_{\N\kp \leq T(\psi)} L_{\kp}(1,\psi)^{\langle \chi, \psi \rangle}, 
\]
where $T(\psi) \geq \max\{ \log q(\psi) , e\}$ for every irreducible character $\psi \in \mathrm{Irr}(G)$. We will make the choice $T(\mathbf{1}_G) = \log(eD_k)$ and $T(\psi) = \log q(\chi)$ for every nontrivial $\psi$, yielding
\begin{align*}
\kappa(\chi) 
& \asymp_{[k:\Q],|G|,\chi(1)} \frac{1}{(\log \log eD_k)^{\langle \chi, \mathbf{1}_G \rangle} }  \prod_{\N\kp \leq \log e D_k} |L_{\kp}(1,\mathbf{1}_G)|^{\langle \chi, \mathbf{1}_G\rangle}   \prod_{\psi \neq \mathbf{1}_G} \prod_{\N\kp \leq \log q(\chi) } |L_{\kp}(1,\psi)|^{\langle \chi, \psi\rangle}. 
\end{align*}
To estimate the remaining product, observe that \cref{prop:small-primes} holds with the weaker conditions $T_n \geq \log q(\psi_n)$ for $1 \leq n \leq N$ by replacing \cref{lem:mertens} with \eqref{eqn:mertens-grh}. Combining these applications exactly as we do in the proof of \cref{thm:cond} in \cref{sec:proofs-main}, it follows that 
\begin{align*}
\kappa(\chi) 
& \ll_{[k:\Q],|G|,\chi(1)} \frac{(\log\log e D_k)^{\widetilde{\chi}(1) - \chi(1) + \langle \chi, \mathbf{1}_G \rangle}}{(\log \log eD_k)^{\langle \chi, \mathbf{1}_G \rangle} }   (\log\log q(\chi))^{\chi(1)-\langle \chi,\mathbf{1}_G \rangle} 
\end{align*}
and similarly for the lower bound. This completes the proof. \qed 

\bibliographystyle{alpha}
\bibliography{references}	

\end{document}